\documentclass[12pt]{amsart}
\usepackage[margin=1in]{geometry}

\usepackage{graphpap}
\usepackage{amsmath}
\usepackage{cases}
\usepackage{amssymb}
\usepackage{setspace}

\usepackage[full]{textcomp}

\usepackage{booktabs}

\usepackage{xcolor}
\usepackage{color}

\usepackage{comment}
\usepackage{breakurl}

\usepackage{mathrsfs}
\usepackage{booktabs}
\usepackage{mathtools}

\usepackage{tikz}
\usepackage{amsthm}                

\usepackage{microtype}

\usepackage{hyperref}
\hypersetup{pdfstartview=}

\addtocontents{toc}{\protect\hypersetup{hidelinks}}

\newcommand{\noi}{\noindent}

\newtheorem{thm}{Theorem}[section]

\newtheorem{theorem}[thm]{Theorem}

\newtheorem{proposition}[thm]{Proposition}
\newtheorem{lem}[thm]{Lemma}

\newtheorem{remarque}[thm]{Remark}
\newtheorem{exemple}[thm]{Exemple}
\numberwithin{equation}{section}

\newtheorem{prop}[thm]{Proposition}
\newtheorem{cor}[thm]{Corollary}

\newtheorem{assumption}[thm]{Assumption}
\theoremstyle{remark}
\newtheorem{remark}[thm]{Remark}
\newtheorem{rmk}[thm]{Remark}

\definecolor{darkgreen}{rgb}{0.1,0.7,0.1}
\definecolor{darkred}{rgb}{0.7,0.1,0.1}
\definecolor{darkblue}{rgb}{0,0,0.7}
\colorlet{symbols}{blue!90!black}
\colorlet{testcolor}{green!60!black}

\usetikzlibrary{shapes.misc}
\usetikzlibrary{shapes.symbols}
\usetikzlibrary{snakes}
\usetikzlibrary{decorations}
\usetikzlibrary{decorations.markings}

\tikzset{
	dot/.style={circle,fill=black,inner sep=0pt, minimum size=1mm},
}

\makeatletter
\def\DeclareSymbol#1#2#3{\expandafter\gdef\csname MH@symb@#1\endcsname{\tikz[baseline=#2,scale=0.15,draw=symbols]{#3}}\expandafter\gdef\csname MH@symb@#1s\endcsname{\scalebox{0.5}{\tikz[baseline=#2,scale=0.15,draw=symbols]{#3}}}}
\def\<#1>{\csname MH@symb@#1\endcsname}
\makeatother

\DeclareSymbol{1}{0}{\draw[white] (-.6,0) -- (.6,0); \draw (0,-0.3)  -- (0,1.8) node[dot] {};}

\DeclareMathSymbol{\sminus}{\mathbin}{AMSa}{"39}

\makeatletter 
\newcommand{\cov}{{\operator@font cov}}
\newcommand{\var}{{\operator@font var}}
\newcommand{\corr}{{\operator@font corr}}
\newcommand{\diam}{{\operator@font diam}}
\newcommand{\Av}{{\operator@font Av}}
\newcommand{\trig}{{\operator@font trig}}
\newcommand{\Law}{{\operator@font Law}}
\newcommand{\Enh}{{\operator@font Enh}}
\newcommand{\EEnh}{\overline {\operator@font Enh}}
\makeatother

\newcommand{\eps}{\varepsilon}

\renewcommand{\d}{\partial}

\newcommand{\cC}{\mathcal{C}}
\newcommand{\dD}{\mathcal{D}}

\newcommand{\fF}{\mathcal{F}}

\newcommand{\iI}{\mathcal{I}}

\newcommand{\nN}{\mathcal{N}}
\newcommand{\oO}{\mathcal{O}}

\newcommand{\qQ}{\mathcal{Q}}

\newcommand{\sS}{\mathcal{S}}

\newcommand{\yY}{\mathcal{Y}}
\newcommand{\zZ}{\mathcal{Z}}

\newcommand{\lfl}{\left\lfloor }  
\newcommand{\rfl}{\right\rfloor} 

\newcommand{\E}{\mathbf{E}}

\newcommand{\N}{\mathbf{N}}
\renewcommand{\P}{\mathbf{P}}
\newcommand{\Q}{\mathbf{Q}}
\newcommand{\R}{\mathbf{R}}
\newcommand{\T}{\mathbf{T}}

\newcommand{\Z}{\mathbf{Z}}

\newcommand{\1}{\mathbf{1}}
\newcommand{\one}{\mathbf{1}}

\renewcommand{\k}{\mathbf{k}}

\def\ov{\overline} 
\def\11{{\rm 1~\hspace{-1.4ex}l} }

\def\HH{\mathbb H}

\newcommand{\err}{\text{err}}

\def\scal#1{\langle#1\rangle}

\begin{document}
	\title{Weak universality results for a class of nonlinear wave equations}
	\author{Chenmin Sun}
	\address{CNRS and
		Universit\'e Paris Est Cr\'eteil, UMR 8050 du CNRS}
	\email{chenmin.sun@cnrs.fr}
	
	\author{Nikolay Tzvetkov}
	\address{
		Universit\'e de Cergy-Pontoise, Cergy-Pontoise, F-95000,UMR 8088 du CNRS}
	\email{nikolay.tzvetkov@u-cergy.fr}

	\author{Weijun Xu}
	\address{Beijing International Center for Mathematical Research, Peking University, Beijing, 100871, China}
	\email{weijunxu@bicmr.pku.edu.cn}
	
	\date{\today}
	
	\maketitle

	\begin{abstract} 
		We study the weak universality of the two-dimensional fractional nonlinear wave equation. For a sequence of Hamiltonians of high-degree potentials scaling to the fractional $\Phi_2^4$, we first establish a \emph{sufficient and almost necessary} criteria for the convergence of invariant measures to the fractional $\Phi_2^4$. Then we prove the convergence result for the sequence of associated wave dynamics to the (renormalized) cubic wave equation. Our constraint on the fractional index is independent of the degree of the nonlinearity. This extends the result of Gubinelli-Koch-Oh [Renormalisation of the two-dimensional stochastic nonlinear wave equations, Trans. Amer. Math. Soc. 370 (2018)] to a situation where we do not have a local Cauchy theory with highly supercritical nonlinearities.  
	\end{abstract}


	\setcounter{tocdepth}{1}
	\microtypesetup{protrusion=false}
	\tableofcontents
	\microtypesetup{protrusion=true}

	\section{Introduction}

	
	\subsection{From microscopic to macroscopic wave dynamics}
	
	The aim of this article is to study the macroscopic behaviour of the weakly interacting waves of the type
	\begin{equation} \label{eq:wave_micro}
		\begin{cases}
			&\d_t^2 \tilde{u} + |\nabla|^{2\alpha} \tilde{u} + N^{-\theta} \widetilde{\Pi}_N V'(\tilde{u}) = 0\;, \quad (t,x) \in \R \times \T_N^2\;,\\
			&\tilde{u}(0,\cdot) = \tilde{\phi}\;, \quad (\d_t \tilde{u})(0,\cdot) = \tilde{\psi}\;,
		\end{cases}
	\end{equation}
	where $\T_N^2 = (\R / 2 \pi N \Z)^{2}$ is the two dimensional torus of side length $2 \pi N$, $V$ is an even polynomial satisfying certain structural conditions specified below, and $\widetilde{\Pi}_N$ is the Fourier projection operator on $\T_N^2$ such that
	\begin{equation*}
		\widetilde{\Pi}_N f(x) = \sum_{|k|\leq N}(\mathcal{F}_Nf)(k)\mathrm{e}^{i\frac{k\cdot x}{N}},\quad (\mathcal{F}_Nf)(k)=\frac{1}{(2\pi N)^2}\int_{\T_N^2}f(y)\mathrm{e}^{-i\frac{y\cdot k}{N}}dy.
	\end{equation*}
    The differential operator $|\nabla|^{\gamma}$ acts on functions on torus of side length $N$ as
    \begin{equation*}
    \mathcal{F}_N(|\nabla|^{\gamma} f)(k) :=  \Big| \frac{ k}{N} \Big|^{\gamma} (\mathcal{F}_Nf)(k)\;.
    \end{equation*}
    Here in the microscopic model, we take $\gamma = 2\alpha$ and $L = 2 \pi N$. The initial data $\tilde{\phi}$ and $\tilde{\psi}$ are two random functions given by
	\begin{equation*}
		\tilde{\phi}_N(x) = \frac{1}{2 \pi N^{1-\alpha}} \sum_{|k| \leq N} \frac{g_k}{\sqrt{1+|k|^{2\alpha}} } \; e^{i \frac{k \cdot x}{N}}\;, \qquad \tilde{\psi}_N(x) = \frac{1}{2 \pi N} \sum_{|k| \leq N} h_k \; e^{i \frac{k \cdot x}{N}}\;,
	\end{equation*}
	where $\{g_k\}$ and $\{h_k\}$ are independent standard complex Gaussian with $g_{-k} = \overline{g_k}$ and the same for $h_k$. This type of initial condition is natural since the Gaussian measure it induces is invariant under the perturbed linear evolution above (with the differential operator $|\nabla|^{2\alpha}$ replaced by $\frac{1}{N^{2\alpha}} + |\nabla|^{2\alpha}$ and without nonlinear interaction). 
	
	\begin{rmk}
		The initial data is, very roughly speaking, of the type
		\begin{equation*}
			\frac{1}{2\pi N} \sum_{|k| \leq N} \rho(k / N) g_k(\omega) e^{i \frac{k \cdot x}{N}}
		\end{equation*}
	    for suitable function $\rho: \R^2 \rightarrow \R$. In our case, $\rho(x) = \frac{1}{\scal{x}^{\alpha}}$ for the initial position, and $\rho(x) \equiv 1$ for the initial velocity. Although natural from the invariance of the perturbed linear dynamics, we should also note that our choice is also very restrictive relating to the support of the corresponding Gibbs measure. 
	\end{rmk}
	
	Note that $\tilde{\phi}$ has a stationary Gaussian distribution with $\widetilde{\phi}_N(x) \sim \nN(0, \sigma_N^2)$, where
	\begin{equation} \label{eq:sigma_N}
		\sigma_N^2 = \frac{1}{4 \pi^2 N^{2(1-\alpha)}} \sum_{|k| \leq N} \frac{1}{1+|k|^{2\alpha}} = \underbrace{\frac{1}{4 \pi^2} \int_{|\xi| \leq 1} \frac{1}{|\xi|^{2\alpha}} {\rm d} \xi}_{\sigma^2} \; + \; \oO(N^{-2(1-\alpha)})\;.
	\end{equation}
	Let $\sigma^2$ be defined as above, $\tilde{\mu}$ be the law of $\nN(0,\sigma^2)$, and
	\begin{equation*}
		\scal{V}(z):= \int_{\R} V(z+y) \tilde{\mu}({\rm d}y)
	\end{equation*}
	be the average of $V$ under $\tilde{\mu}$. Our main assumption on $V$ is the criticality and positivity of its averaged version $\scal{V}$. 
	
	\begin{assumption} \label{as:V}
		$V$ is an even polynomial of degree $2m \geq 4$ with the form
		\begin{equation*}
			V(z) = \sum_{j=0}^{2m} a_j z^{2j}\;.
		\end{equation*}
		Furthermore, we assume
		\begin{enumerate}
			\item $z=0$ is a bifurcation point of $\scal{V}$ in the sense that $\scal{V}''(0) = 0$. 			
			\medskip		
			\item $\scal{V}(z) - \scal{V}(0) > 0$ for all $z \neq 0$. 
		\end{enumerate}
	\end{assumption}
	The averaged version $\scal{V}$ has the expression
		\begin{equation*}
			\scal{V}(z) = \sum_{j=0}^{m} \overline{a}_j z^{2j}
		\end{equation*}
		with 
		\begin{equation} \label{eq:potential_coefficients_limit}
			\overline{a}_{j} = \frac{1}{(2j)!} \E \big[ V^{(2j)}\big( 
			{\mathcal N}
			(0,\sigma^2) \big) \big] = \frac{1}{(2j)!} \sum_{k=j}^{m} \frac{(2k)!}{(2k-2j)!!} \cdot a_k \cdot \sigma^{2(k-j)}\;.
		\end{equation}
		Hence, Condition (1) above is equivalent to say that $\overline{a}_1 = 0$. Since the renormalisation term in the wave dynamics and the measures are constant multiples of $\overline{a}_1 N^{2(1-\alpha)} u_N$ and $\overline{a}_{1} N^{2(1-\alpha)} \phi_N^{\diamond 2}$ respectively, Condition (1) guarantees that the divergent parts in various terms are cancelled out automatically, and there is no need to subtract the renormalisation by hand. With Condition (1), Condition (2) is then equivalent to the following positivity condition:
	\begin{align} \label{eq:average_V_positive}
	\sum_{j=2}^{m} \overline{a}_j z^{2(j-2)} > 0\;, \qquad \forall z \in \R\;.
\end{align}

\begin{exemple}
If we fix $a_2>0$, ... ,$a_m>0$, we can find $a_1<0$ such that our assumptions on $V$ are satisfied.  For example 
$$
V(z)=z^6-45 \sigma^2 z^2
$$
satisfies the assumptions. We can also find $V\geq 0$ such that our assumptions are satisfied. 
\end{exemple}
	
	Our aim is to investigate the influence of the microscopic weak non-linear interaction to the macroscopic behaviour of $\tilde{u}$ under the above assumption on $V$. For $\T^2 = (\R / 2 \pi \Z)^2$, define the macroscopic process $u_N$ on $\R \times \T^2$ by
	\begin{equation*}
		u_{N}(t,x) := N^{1-\alpha} \tilde{u}(N^\alpha t, Nx)\;, \quad (t,x) \in \R \times \T^2\;.
	\end{equation*} 
	It satisfies the equation
	\begin{equation} \label{eq:wave_macro}
		\d_t^2 u_N + |\nabla|^{2\alpha} u_N + N^{1+\alpha-\theta} \Pi_N V'(u_N / N^{1-\alpha}) = 0\;, \quad (t,x) \in \R \times \T^2
	\end{equation}
	with initial data
	\begin{equation} \label{eq:initial_data_macro}
		(u_N,\partial_tu_N)(0,x) = (\phi_N(x),\psi_N(x)) = \frac{1}{2\pi} \Big(\sum_{|k| \leq N} \frac{g_k}{\sqrt{1+|k|^{2\alpha}}} e^{i k \cdot x},\;  \sum_{|k| \leq N} h_k e^{ik \cdot x}\Big)\;,
	\end{equation}
where $\Pi_N$ is the Fourier projector on the unit tori:
\begin{align}\label{PiN}  \widehat{\Pi_Nf}(k)=\mathbf{1}_{|k|\leq N}\widehat{f}(k),\quad \widehat{f}(k)=(\mathcal{F}_1f)(k)=\frac{1}{(2\pi)^2}\int_{\T^2}f(y)\mathrm{e}^{-ik\cdot y}dy.
\end{align}
	In order for $u_N$ to converge to a cubic equation, one necessarily sets $\theta = 4 \alpha -2$ and hence $1+\alpha - \theta = 3(1-\alpha)$. 


	\subsection{The macroscopic model}
	Let $\T^2 = (\R / 2 \pi \Z)^{2}$ be the two dimensional torus. For every $N>0$, let $\Pi_N$ be the Fourier projection operator on the unit tori introduced in \eqref{PiN}. For $\alpha \in (\frac{3}{4}, 1)$, let $\mu = \mu^\alpha$ be the probability measure on $\dD'(\T^2)$ (the space of distributions on $\T^2$) with covariance operator $(1+|\nabla|^{2\alpha})^{-1}$, and $\mu'$ be the white noise on $\T^2$. Equivalently, the Gaussian measures $\mu^\alpha$ and $\mu'$ are induced by the random functions
	\begin{equation*}
		\phi = \frac{1}{2 \pi} \sum_{k\in\Z^2} \frac{g_k}{\sqrt{1+|k|^{2\alpha}}} \mathrm{e}^{ik\cdot x}\;, \quad \psi = \frac{1}{2\pi} \sum_{k \in \Z^2} h_k e^{i k \cdot x}
	\end{equation*}
    respectively, where $\{g_k\}_{k \in \Z^2}$ is a collection of centered complex Gaussian random variables such that
	\begin{equation*}
		g_{-k} = \overline{g_{k}}\;, \qquad \E |g_{k}|^2 = 1\;, \qquad \E (g_{k}^2) = 0\;, 
	\end{equation*}
	and otherwise independent, and the same for $\{h_k\}$. Since $\alpha$ will be a fixed parameter throughout the article, we simply write $\mu = \mu^\alpha$. 
			 
	Let $\mu_N:= \mu \circ \Pi_N^{-1}$ and $\mu'_N = \mu' \circ \Pi_N^{-1}$ be the marginals of $\mu$ and $\mu'$ on frequencies up to $N$. Hence, the initial data of the macroscopic wave dynamics \eqref{eq:wave_macro} are distributed according to $\mu_N \otimes \mu'_N$. Let $\widetilde{\sigma}_N^2$ be the variance of $\phi$ under $\mu_N$, which is invariant under translations and hence $\widetilde{\sigma}_N^2$ does not depend on the spatial variable $x$. In fact, a direct computation shows
	\begin{align}\label{eq:sigma_N_tilde}
		\widetilde{\sigma}_N^2:= \E^{\mu} |\Pi_N \phi|^2 = \frac{1}{4 \pi^2} \sum_{k \in \Z^2,|k|\leq N} \frac{1}{1+|k|^{2\alpha}} = \underbrace{(\sigma^2 + \err_N)}_{=: \sigma_N^2} \cdot N^{2(1-\alpha)}\;,
	\end{align}
	where $\sigma_N^2$ and
	\begin{equation} \label{eq:sigma}
		\sigma^2 = \frac{1}{4\pi^2} \int_{|\xi|\leq 1} \frac{1}{|\xi|^{2\alpha}} {\rm d}\xi\;
	\end{equation}
	are as defined in \eqref{eq:sigma_N}, and $\err_N = \oO(N^{-2(1-\alpha)})$ as $N \rightarrow +\infty$. 
			
	Now, let $V$ be an even polynomial satisfying Assumption~\ref{as:V}. For every $N \in \N$, let
	\begin{equation} \label{eq:potential_renormalised}
		V_{N}(\varphi) = N^{4(1-\alpha)} V(\varphi / N^{1-\alpha}) \;,
	\end{equation}
	and we have
	\begin{equation} \label{eq:potential_renormalised_wick}
		V_{N}(\varphi) = \sum_{j=1}^{m} \overline{a}_{j,N} N^{-(2j-4)(1-\alpha)} H_{2j}(\varphi; \widetilde{\sigma}_N^2)\;, 
	\end{equation}
	where $H_{k}(\cdot, \sigma^2)$ is the $k$-th Hermite polynomial with leading coefficient $1$ and variance $\sigma^2$.  The coefficients $\overline{a}_{j,N}$ can be explicitly computed as
	\begin{equation} \label{eq:potential_coefficients_wick}
		\overline{a}_{j,N} = \frac{1}{(2j)!} \E \big[ V^{(2j)}\big( \nN(0,\sigma_N^2) \big) \big]\;.
	\end{equation}
	For every $j$, we have $\overline{a}_{j,N} \rightarrow \overline{a}_j$ as $N \rightarrow +\infty$, where $\overline{a}_j$ are as given in \eqref{eq:potential_coefficients_limit}. Furthermore, the following slightly more delicate relation holds. 
			
\begin{prop}\label{convergeneclinear} 
Assume that $\alpha\in\big(\frac{1}{2},1\big)$. There exists an absolute constant $\lambda_0\in\R$, such that as $N\rightarrow\infty$,
$$ \ov{a}_{1,N}=\ov{a}_1+\lambda_0N^{-2(1-\alpha)}+O(N^{-1})+O(N^{-4(1-\alpha)}).
$$
\end{prop}
\begin{proof}
See Appendix~\ref{appendix:1orderconstant}.
\end{proof}

\subsection{Wave dynamics}

Our first main result concerns the behavior of the macroscopic wave-dynamics as $N\rightarrow\infty$. In this part, we always assume that $V$ verifies Assumption \ref{as:V} and denote $\lambda:= 4\overline{a}_2 > 0$. The theorem is stated as follows. 

\begin{theorem} \label{thm:dynamics}
	Suppose that $\alpha \in (\frac{8}{9}, 1)$. Let $\sigma < \alpha-1$ and suppose that $V$  satisfies Assumption \ref{as:V} with $\lambda:=4\ov{a}_2>0$. Let $u_N$ be the solution of 
	\begin{equation*}
		\d_t^2 u_N + 
		|\nabla|^{2\alpha}
		u_N + \Pi_N V_N' (u_N) = 0,
	\end{equation*}
	with initial data 
	\begin{equation*} \label{donnee_bis}
		(u_N, \d_t u_N)|_{t=0} = \frac{1}{2\pi}\,  
		\sum_{|k| \leq N } 
		\Big(
		\frac{g_k(\omega)}{\sqrt{1+|k|^{\alpha}}} \,\,\mathrm{e}^{i k\cdot x}, \; h_k(\omega) \,\, \mathrm{e}^{i k\cdot x }\Big)\,.
	\end{equation*}
	Then solutions of (with $\lambda_0\in\R$ given in Proposition \ref{convergeneclinear})
	\begin{equation*}
		\d_t^2 v_N + |\nabla|^{2\alpha}  v_N +2\lambda_0v_N+  \lambda\Pi_N (  (v_N)^3-3\widetilde{\sigma}_N^2 v_N ) = 0
	\end{equation*}
	with initial data \eqref{donnee_bis}  converge almost surely in the sense of distribution on $\R\times \T^2$, as $N\rightarrow \infty$ and satisfy 
	\begin{equation*}
		\lim_{N\rightarrow\infty}\|u_N - v_N\|_{\cC([-T,T], H^{\sigma}(\T^2))} = 0,\, \forall\, T>0.
	\end{equation*}
\end{theorem}

\begin{remark}
	We have a more detailed convergence statement by decomposing $u_N$ (and also $v_N$) into a random term with low regularity and a smoother contribution. The latter converges in positive Sobolev norms. See Propositions~\ref{th:wave_cubicdetailed} and ~\ref{convergenceN} for precise statements. 
	
	The restriction $\alpha>\frac{8}{9}$ is technical and can hopefully be improved using recently developed methods (\cite{Bringmann1, 2d_NLS_full, random_tensors}). However, this is not in the objective of this work. Instead we emphasize that our range of $\alpha$ is \emph{independent} of the degree $2m$ of the potential $V$. Indeed, the Cauchy problem \eqref{introwave_N} without the negative powers of $N$ in higher nonlinearities in $V$ is highly supercritical\footnote{For large $m$, this is even supercritical with respect to the probabilistic scaling, a notion introduced in \cite{2d_NLS_full, random_tensors}.}. What saves us here is the truncation $\Pi_N$ in frequency space and the negative power of $N$ in front of the high-power nonlinearity. The same situation appears in Hairer-Quastel \cite{KPZ_HQ} for the KPZ equation (though in a different setup where the problem is the singularity of the driving noise instead of the initial data). 
\end{remark}

\begin{remark}
	The theorem still holds true if the sharp cutoff in the truncation is replaced by a smoother cutoff with a sufficiently fast decay smooth function. The constant $\lambda$ in the final statement then will depend on the actual cutoff function. 
\end{remark}

\subsection{The Gibbs measures}

In order to prove Theorem~\ref{thm:dynamics}, we re-write the macroscopic model \eqref{eq:wave_macro} as
\begin{equation} \label{eq:wave_macro_new}
	\d_t^2 u_N + (1 + |\nabla|^{2\alpha}) u_N + \Pi_N \big( V_N'(u_N) - u_N \big) = 0\;,
\end{equation}
still with initial data \eqref{eq:initial_data_macro}. We add a mass term in the linear part in order to control the free evolution of the zero-th Fourier mode, and modified the nonlinear term to compensate the change. In fact, without the mass term, the zero-th mode will grow in time under the linear evolution. Let
\begin{equation*}
	\widetilde{V_N}(\varphi) := V_{N}(\varphi) - \frac{1}{2} \big( \varphi^2 - \widetilde{\sigma}_N^2 \big)\;,
\end{equation*}
and let $\nu_N$ be the probability measure given by
	\begin{equation}\label{eq:nu_N}
		\nu_N({\rm d} \phi) = \frac{1}{\zZ_N} e^{-\int_{\T^2} \widetilde{V_N}(\phi) {\rm d}x} \mu_N({\rm d} \phi)\;.
	\end{equation}
The measure $\nu_N$ is well defined as long as $a_m>0$, and $\nu_N \otimes \mu'_N$ is invariant under the dynamics \eqref{eq:wave_macro_new}. If $\lambda := \overline{a}_2 > 0$, then the measure
	\begin{equation*}
		\nu({\rm d} \phi) = \frac{1}{\zZ} e^{-\lambda \int_{\T^2} \phi^{\diamond 4} {\rm d}x + \frac{1}{2} \int_{\T^2} \phi^{\diamond 2} {\rm d}x} \mu({\rm d} \phi)
	\end{equation*}
is also well-defined, where $\phi^{\diamond k}$ denotes the $k$-th Wick power of $\phi$ with respect to the Gaussian structure induced by $\mu$. $\nu$ is known as the fractional $\phi^4_2$ with exponent $\alpha$. See Section \ref{renormalisation} for the precise definition. 

\begin{rmk}
	Note that the measure $\nu$ has an additional quadratic term on the exponential with the opposite sign compared to the usual fractional $\phi^4_2$. This is because we define the Gaussian measure $\mu$ to have covariance $(1 + |\nabla|^{2\alpha})^{-1}$. Indeed, the measure $\nu$ is the same with the quadratic term removed if the reference Gaussian measure has covariance $|\nabla|^{-2\alpha}$ and $0$-mode being a $\nN(0,1)$ random variable independent of all other modes. 
\end{rmk}

Let $\mu'$ be the white noise measure on $\T^2$, and define the measures $\vec{\mu}$, $\vec{\nu}_N$ and $\vec{\nu}$ by
\begin{equation*}
	\vec{\mu} := \mu \otimes \mu'\;, \qquad \vec{\nu}_N:= \nu_N \otimes \mu'_N\;, \qquad \vec{\nu}:= \nu \otimes \mu'\;.
\end{equation*}
More precisely, writing $\vec{\phi} = (\phi, \phi')$, we have
\begin{equation*}
	\vec{\nu}_N({\rm d} \vec{\phi}) = \nu_N({\rm d} \phi) \mu'_N({\rm d} \phi') = \zZ_N^{-1} e^{- \int_{\T^2} \widetilde{V_N}(\phi) {\rm d}x} \underbrace{\mu_N({\rm d} \phi) \mu'_N({\rm d}\phi')}_{\vec{\mu}_N({\rm d} \vec{\phi})}\;,
\end{equation*}
and
\begin{equation*}
	\vec{\nu}({\rm d} \vec{\phi}) = \nu({\rm d} \phi) \mu'({\rm d} \phi') = \zZ^{-1} e^{- \lambda \int_{\T^2} \phi^{\diamond 4} {\rm d}x} \underbrace{\mu({\rm d} \phi) \mu'({\rm d}\phi')}_{\vec{\mu}({\rm d} \vec{\phi})}\;, 
\end{equation*}
where the values of $\zZ_N$ and $\zZ$ are the same as before. 
The equation \eqref{eq:wave_macro_new} can be written as a Hamiltonian system for $\vec{u}_N := (u_N, \d_t u_N)$ as
\begin{equation} \label{introwave_N}
	\d_t \begin{pmatrix}
		u_N \\ \d_t u_N
	\end{pmatrix}
	=
	\begin{pmatrix}
		0 & 1\\
		-1 & 0
	\end{pmatrix}
	\frac{\d \mathcal{E}_N}{\d (u_N, \d_t u_N)}\;,
\end{equation}
where the Hamiltonian is given by
\begin{equation*}
	\mathcal{E}_N(f,g) = \frac{1}{2} \Big( \scal{|\nabla|^{2\alpha} f, f}_{L^2} + \scal{g, g}_{L^2} \Big) + \int_{\T^2} V_N (\Pi_N f) {\rm d}x\;.
\end{equation*}
For every $N$, the probability measure $\vec{\nu}_N$ is invariant under the above Hamiltonian dynamics. Theorem~\ref{th:main_measure} implies that $\vec{\nu}_N \otimes \mu_N^{\perp} \otimes (\mu'_N)^{\perp}$ converges to $\vec{\nu}$ in the sense that the density with respect to $\vec{\mu}$ converges in $L^p(\vec{\mu})$ for every $p \geq 1$. The measures $\vec{\mu}$ and $\vec{\nu}$ are supported on
\begin{equation*}
	\mathcal{H}^{-(1-\alpha) \sminus}(\T^2) := H^{-(1-\alpha) \sminus}(\T^2) \times H^{-1 \sminus}(\T^2)\;,
\end{equation*}
where
\begin{equation*}
	H^{\gamma \sminus} := \bigcap_{\eps>0} H^{\gamma-\eps}\;.
\end{equation*}
The invariance of $\nu_N \otimes \mu_N'$ under the dynamics \eqref{eq:wave_macro_new} is an essential ingredient in the proof of Theorem~\ref{thm:dynamics}. In addition, convergence of the measures itself may of independent interest. 

\begin{rmk}
	We would like to emphasize that the invariance of $\nu_N \otimes \mu_N'$ under the dynamics \eqref{eq:wave_macro_new} is used in two different ways. The first one is that it gives key a priori bounds for truncated dynamics (for fixed $N$). Second, the convergence of the invariant measures to a limiting measure (as stated in Theorem~\ref{th:main_measure} below) and the invariance of the limiting measure under the limiting dynamics allows us to pass from local to global in time convergence. 
\end{rmk}

\subsection{Convergence of the measures}

We now state our result on the convergence of the Gibbs measures. For convenience, we introduce another measure $\overline{\nu}_N$ by
\begin{equation*}
	\overline{\nu}_N ({\rm d} \phi) := \nu_N \otimes \mu_N^{\perp} = \frac{1}{\zZ_N} e^{-\int_{\T^2} \widetilde{V_N} (\Pi_N \phi) {\rm d}x} \mu ({\rm d}\phi)\;,
\end{equation*}
where the normalisation constant $\zZ_N$ is the same as the one in \eqref{eq:nu_N}. For every $p \geq 1$, define
	\begin{equation*}
		\zZ_{N}^{(p)} := \E^{\mu} \Big[ e^{-p \int_{\T^2} \widetilde{V_N}(\Pi_N \phi) {\rm d}x} \Big]\;.
	\end{equation*}
	Then $\zZ_{N} = \zZ_{N}^{(1)}$. Our first theorem is the following:
	
	\begin{thm} \label{th:main_measure}
		Let $\alpha \in (\frac{3}{4}, 1)$. Suppose that $V$ verifies Assumption \ref{as:V}. 
	 Then for every $p \geq 1$, we have
		\begin{equation*}
			\sup_{N} |\log \zZ_N^{(p)}| < +\infty\;
		\end{equation*}
		Furthermore, $\lambda := \overline{a}_2 > 0$, and
		\begin{equation*}
			\E^{\mu} \bigg| e^{-\int_{\T^2} \widetilde{V_{N}}(\Pi_N \phi) {\rm d}x} - e^{- \lambda \int_{\T^2} \phi^{\diamond 4} {\rm d}x + \frac{1}{2} \int_{\T^2} \phi^{\diamond 2} {\rm d}x} \bigg|^{p} \rightarrow 0\;
		\end{equation*}
		for every $p \geq 1$. Hence, $\overline{\nu}_N$ converges to the fractional $\phi^4_2$ measure $\nu$ in the sense that the densities with respect to $\mu$ converge in $L^{p}(\mu)$. 
	\end{thm}
	
	\begin{rmk}
		The restriction $\alpha>\frac{3}{4}$ is natural in the sense that in this range, one can define the $\phi^4$ measure $\nu$ by an absolutely continuous density with respect to the Gaussian measure $\mu$. The fourth Wick power $\phi^{\diamond 4}$ fails to exist under $\mu$ when $\alpha = \frac{3}{4}$, in which case one expects to end up with a measure (after further renormalisations) that is mutually singular with respect to $\mu$. 
	\end{rmk}

The next proposition says that \eqref{eq:average_V_positive} is actually almost necessary for the main theorem.

\begin{prop} \label{pr:measure_counter_eg}
	If there exists $\theta \in \R$ such that $\sum_{j=1}^{m} \overline{a}_2 \theta^{2(j-2)} < 0$, then there exists $c>0$ such that $\log \zZ_N > c N^{4(1-\alpha)}$ for all $N \in \N$. As a consequence, the densities $\frac{{\rm d} \bar{\nu}_N}{{\rm d} \mu}$ do not converge in $L^1(\mu)$. 
\end{prop}

\subsection{Comparison with parabolic equations and other dispersive models}

This type of weak universality was first studied by Hairer-Quastel (\cite{KPZ_HQ}) in deriving the KPZ equation from a large class of microscopic growth models. It has later been extended in various directions in the setting of parabolic singular stochastic PDEs (\cite{KPZ_general_nonlinear, phi4_poly, phi4_non_Gaussian, phi4_general_nonlinear, phi4_general_smoothing}). A key feature in this type of this problem is that every term in the expansion of the nonlinearity has the same size --- and hence the constant $\lambda$ of this limiting equation depends on the whole nonlinearity rather than the naive guess of the corresponding power only. As far as we know, our Theorem \ref{thm:dynamics} is the first one for dispersive models fitting in this situation. 

Technically, one difference between dispersive and parabolic equations is the lack of $L^\infty$ based estimates in the dispersive setting. Hence, the heuristic reasoning that negative powers of $N$ balance out high powers of singular objects needs more involved justification with the help of dispersive tools. A second technical difference lies in the globalisation argument. In the parabolic setting, the global-in-time convergence follows from the global well-posedness of the limiting equation and stability. However in the current dispersive setting, even though the limiting equation is globally well-posed, the stability properties are not good enough here, and we need to make an essential use of invariant measure to get global convergence. 

Note that our techniques can be used to extend the weak universality result of  Gubinelli-Koch-Oh for the 2D stochastic nonlinear wave equation to the stochastic nonlinear fractional wave equation with space-time white noise, formally written as 
\begin{equation*}
	\partial_t^2u+|\nabla |^{2\alpha}u+\partial_tu+\lambda u^{\diamond 3}= \xi,\quad (t,x) \in \R^+ \times \T^2
\end{equation*}
when $\alpha>\frac{8}{9}$. The weak universality result of Gubinelli-Koch-Oh  is a consequence of the almost sure global well-posedness for the two-dimensional nonlinear wave equation ($\alpha=1$) with \emph{any order nonlinearity}, while for the fractional wave equation with $\alpha<1$, the situation is radically different.  


\subsection{Notations and conventions}

We fix the parameter $\alpha \in (\frac{8}{9}, 1)$ throughout this article. In the Gibbs measure part, we relax its range to $\alpha \in (\frac{3}{4}, 1)$. For $z \in \R$, we write $\scal{z} := (1 + |z|^{2\alpha})^{\frac{1}{2\alpha}}$, and write
\begin{equation*}
	\scal{\nabla} := (1 + |\nabla|^{2\alpha})^{\frac{1}{2\alpha}}\;.
\end{equation*}
We use the short symbol $\dD := \scal{\nabla}$ throughout this article (and hence $\dD^\alpha = \scal{\nabla}^{\alpha}$). The estimates $X\lesssim Y$ ($X\gtrsim Y$) stand for $X\leq CY$ ($X\geq C'Y$) for some unrelated constants $C,C'>0$. We denote by $X\sim Y$ if $X\lesssim Y$ and $X\gtrsim Y$. We also write $X\lesssim_{\eps,\delta} Y$ to emphasize that the constant depends only on parameters $\eps,\delta$.  

Space-time norms are frequently used in the article. For a Banach space $\mathcal{X}$, an interval $I\subset \R$ and $q\in[1,\infty]$, we denote by $L_t^q\mathcal{X}(I)$ the space $L^q(I;\mathcal{X})$. If there is no risk of confusing about the time interval, we will simply write $L_t^q\mathcal{X}$. Banach spaces $\mathcal{X}$ can be Sobolev or Lebesgue spaces for the spatial variable, such as $H^{s}(\T^d), W^{s,p}(\T^d), L^p(\T^d)$, etc. Furthermore, since no local spatial norm will be used, we write $\mathcal{X}_x$ to stand form $\mathcal{X}(\T^d)$. 

Globally reserved parameters: $\alpha\in(\frac{1}{2},1)$, $\beta:=1-\alpha$, $s_0=4\alpha-3$. The even number $2m\in\N$ stands for the degree of the potential $V(z)$. We may specify more restrictive ranges of them in different contexts.  

For parameters $A,B$, the symbol $A\ll B$ means that $B>CA$ for a very large constant $C$, depending on the context.

\subsection{Organization of the article}

This article is organized as follows. In Section~\ref{sec:prelim}, we give some preliminary lemmas on functional inequalities and stochastic estimates. These will be used throughout the article. In Section~\ref{sec:measure_convergence}, we prove the convergence of the measures $\nu_N$ to $\nu$ under Assumption~\ref{as:V} on $V$, and also give evidence to show that this positivity assumption is also almost necessary for the convergence result. Section~\ref{sec:wave_convergence} is devoted to the proof of Theorem~\ref{thm:dynamics}, convergence of the wave dynamics to the cubic wave equation. The appendices collect detailed proofs of some technical lemmas.

\subsection*{Acknowledgement}
	
The authors thank the Hausdorff Research Institute for Mathematics for the hospitality, since the beginning of this work has been done during the program \emph{Randomness, PDEs and Nonlinear Fluctuations}. C.~Sun and N.~Tzvetkov are supported by the ANR grant ODA (ANR-18-CE40- 0020-01).

\section{Preliminaries} \label{sec:prelim}

\subsection{Functional spaces and nonlinear estimates}\label{functionalspaces} 
	
	Let $\varphi\in C_c^{\infty}(\R^d;[0,1])$ be a radial functions such that supp$(\chi)\subset\{\xi:|\xi|\leq \frac{4}{3} \}$, supp$(\varphi)\subset\{\xi: \frac{3}{4} \leq |\xi|\leq \frac{8}{3} \}$. For $j\geq 0$, define $\varphi_j(\xi)=\varphi(2^{-j}\xi)$. Let $\chi: \R^d\rightarrow [0,1]$ be a radial bump function such that
	$$ \chi(\xi)+\sum_{j\geq 0}\varphi_j(\xi)\equiv 1.
	$$ 
	Define the Fourier multiplier $$\mathbf{P}_{-1}=\mathcal{F}_x^{-1}\chi\mathcal{F}_x,\quad \mathbf{P}_j=\mathcal{F}_x^{-1}\varphi_j\mathcal{F}_x,\; j\geq 0.
	$$
	The Besov space $B_{p,q}^{\gamma}(\T^d)$ with indices $\gamma\in\R, 1\leq p,q\leq\infty$ is defined via the norm
	$$  \|f\|_{B_{q,r}^{\gamma}(\T^d) }:=\big\|2^{j\gamma}\|\mathbf{P}_jf\|_{L^q(\T^d)} \big\|_{l_j^p}.
	$$
	In the main part of the article, we use frequently the convention $\mathcal{C}^{\gamma}:=B_{\infty,\infty}^{\gamma}$. For $\gamma\in\R, 1\leq p\leq\infty$, the fractional Sobolev spaces $W^{\gamma,p}(\T^d)$ is defined via the norm
	$$ \|f\|_{W^{\gamma,p}(\T^d)}:= 
	\|\dD^{\gamma}f\|_{L^p(\T^d)}
	$$
	By Littlewood-Paley's square-function theorem, when $\gamma\geq 0$ and $1<p<\infty$, we hvae
	$$ \|f\|_{W^{\gamma,p}(\T^d)}\sim_{\gamma,p} \big\|\|2^{j\gamma}\mathbf{P}_jf\|_{l_j^2} \big\|_{L^p(\T^d)}.
	$$
	
	\begin{lem} [Fractional Leibniz rule]
		\label{le:fractional_Leibniz}
		Let $p \in (1, +\infty)$ and $p_1, p_2, \widetilde{p}_1, \widetilde{p}_2 > 1$ such that
		\begin{equation*}
			\frac{1}{p_1} + \frac{1}{p_2} = \frac{1}{\widetilde{p}_1} + \frac{1}{\widetilde{p}_2} = \frac{1}{p}. 
		\end{equation*}
		Let $\beta \geq 0$. Then there exists $C>0$ depending on all the above parameters such that
		\begin{equation*}
			\|\langle\nabla\rangle^{\beta}(fg)\|_{L^p} \leq C \big( \|\langle\nabla\rangle^{\beta} f\|_{L^{p_1}} \|g\|_{L^{p_2}} + \|f\|_{L^{\widetilde{p}_1}} \|\langle\nabla\rangle^{\beta} g\|_{L^{\widetilde{p}_2}} \big)
		\end{equation*}
		for all $f,g \in \cC^{\infty}(\T^2)$. 
	\end{lem}
	\begin{proof}
		This is \cite[Theorem~1]{fractional_Leibniz}. 
	\end{proof}

	\begin{prop} [General Gagliardo-Nirenberg inequality]
		\label{pr:general_Sobolev}
		Let $p \in (1+,\infty)$, $\beta>0$ and $\theta \in [0,1]$. Let $p_1, p_2 > 1$ and $\beta_1, \beta_2 > 0$ be such that
		\begin{equation*}
			\frac{1}{p} = \frac{\theta}{p_1} + \frac{1-\theta}{p_2} \quad \text{and} \quad \beta = \theta \beta_1 + (1-\theta) \beta_2\;. 
		\end{equation*}
		Then we have
		\begin{equation*}
			\|f\|_{W^{\beta,p}} \lesssim \|f\|_{W^{\beta_1,p_1}}^{\theta} \|f\|_{W^{\beta_2,p_2}}^{1-\theta}
		\end{equation*}
		for all $f \in \cC^{\infty}$. The proportionality constant depends on all the above parameters but is independent of $f$. 
	\end{prop}
	\begin{proof}
		This is the content of \cite[Theorem~1]{Gagliardo_Nirenberg_general}. 
	\end{proof}
	
	\subsection{Strichartz estimate}
	Consider the fractional wave equation on $\R^d$, with $0<\alpha<1$:
	\begin{align}\label{linearwave1} 
		\partial_t^2u+(\dD^{\alpha})^2 u=F.
	\end{align} 
	We say that $(q,r)$ is admissible, if
	$$ \frac{2}{q}\leq d\big(\frac{1}{2}-\frac{1}{r}\big), \quad (q,r,d)\neq (2,\infty,2),
	$$
	and $(q,r)$ is sharp admissible if the equality holds. 
	Denote by
	$$ \gamma_{q,r}:=d\big(\frac{1}{2}-\frac{1}{r}\big)-\frac{\alpha}{q}.$$
	We have the following Strichartz estimate:
	\begin{proposition}[\cite{DD}]\label{Strichartz} 
		Assume that $\alpha\in(0,1)$. Let $(q,r)$ be a sharp admissible pair. For any solution $u$ of \eqref{linearwave1}, we have
		\begin{align}\label{Strichartz1} 
			\|u\|_{L_t^qL_x^r([0,T]\times\T^d)}\leq C_{q,r}\|(u,\partial_tu)|_{t=0}\|_{H^{\gamma_{q,r}}(\T^d)}+C_{q,r}\|F\|_{L_t^1H_x^{\gamma_{q,r}-\alpha}([0,T]\times\T^d)},
		\end{align} 
		where the constant $C_{q,r}$ is independent of $T>0$.
	\end{proposition}
	Note that when $d=2$,  if $(q,r)$ is sharp admissible, then
	$\gamma_{q,r}=\frac{2-\alpha}{q}.
	$
	We will only make use of the Strichartz space $L_t^qL_x^{r}$ for $q$ slightly greater than $2$ in this article. Due to the finite propagation speed for the linear wave when $\alpha<1$, the Strichartz estimate is the same as 
	that in $\R^d$, which follows from a standard stationary phase analysis and a $TT^*$ argument.	In order to be self-contained, we include a proof of Proposition \ref{Strichartz1} in the appendix.
	
	\subsection{Renormalisation and the white-noise functional}\label{renormalisation}
	
	First we recall that the Hermite polynomials $H_k(x;\sigma)$ can be defined via the generating function
	$$ F(t,x;\sigma)=e^{tx-\frac{1}{2}\sigma t^2}=\sum_{k=0}^{\infty}\frac{t^k}{k!}H_k(x;\sigma).
	$$
	It follows that
	\begin{align}\label{Hermite-formula}
		H_k(x;\sigma)=\sum_{j=0}^{\big\lfloor\frac{k}{2} \big\rfloor}\binom{k}{2j}(2j-1)!! (-\sigma)^j x^{k-2j}.
	\end{align}
	When $\sigma=1$, we denote by $H_k(x)=H_k(x;1)$. The relation of $H_k(x,\sigma)$ and $H_k(x)$ is given by
	$$ H_k(x;\sigma)=\sigma^{\frac{k}{2}}H_k\Big(\frac{x}{\sqrt{\sigma}}\Big).
	$$
	Taking derivatives of the generating function, one deduces easily that
	$$ \partial_x^jH_k(x;\sigma)=\frac{k!}{(k-j)!}H_{k-j}(x;\sigma).
	$$
	Furthermore, by the multiplicative property of the generating function: $$F(t,x+y;\sigma_1+\sigma_2)=F(t,x;\sigma_1)\cdot F(t,x;\sigma_2),$$ we have the binomial expansion
	\begin{align}\label{additive}
		H_k(x+y;\sigma_1+\sigma_2)=\sum_{l=0}^k\binom{k}{l}H_l(x;\sigma_1)H_{k-l}(y;\sigma_2).
	\end{align}
	Hermite polynomials can be used to define the Wick-ordered product for real-valued Gaussian random variables. Let $z$ be a real-valued Gaussian random variable generated by $(\widetilde{g}_k)_{k\in\N}$ with $\nu$. Then we define its Wick product as
	\begin{align}\label{Wick}
		z^{\diamond k}:=H_k(z;\nu).
	\end{align}
	From \eqref{additive}, we have for any function $w$,
	$$ H_k(z+w;\nu)=\sum_{l=0}^k H_l(z;\nu)\cdot w^{k-l}.
	$$
	When $w$ represents a deterministic function, sometimes we will also use $(z+w)^{\diamond k}$ to represent $H_k(z+w;\nu)$.
	For independent real-valued Gaussian random variables $z_1, z_2$ generated by $(\widetilde{g}_k)_{k\in\N}$ with variance $\nu_1,\nu_2$, with respectively, we have the binomial expansion:
	\begin{align}\label{binomHermite} 
	& (z_1+z_2)^{\diamond k}:=H_k(z_1+z_2;\nu_1+\nu_2)\notag \\=&\sum_{l=0}^k\binom{k}{l}H_l(z_1;\nu_1)H_{k-l}(z_2;\nu_2)=\sum_{l=0}^k\binom{k}{l}z_1^{\diamond l}\cdot z_2^{\diamond (k-l)}.
	\end{align}

	In order to estimate the regularity of wick-products, it is convenient to use the white-noise functional calculus. Let 
	$$ \xi^{\omega}(x)=\sum_{n\in\Z^d}\widetilde{g}_n(\omega)e^{inx}
	$$ 
	be the real-valued white noise distribution on $\T^d$,
	where $(\widetilde{g}_n)_{n\in\Z}$ is a sequence of complex-valued independent $\mathcal{N}_{\mathbb{C}}(0;1)$ Gaussian random variables on a given probability space $(\Omega,\mathcal{F},\mathbb{P})$, conditioned to $\ov{\widetilde{g}}_{n}=\widetilde{g}_{-n},\forall n\in\Z^d$. We define the white-noise functional  $$W_{(\cdot)}: L^2(\T^d)\rightarrow L^2(\Omega,\mathcal{F},\mathbb{P})$$
	by
	$$ f\mapsto W_f(\omega)=(f,\xi^{\omega})_{L^2(\T^d)}:=\sum_{n\in\Z^d}\widehat{f}(n)g_n(\omega).
	$$ 
	Note that for any $f,h\in L^2(\T^d)$, we have
	$$ \mathbb{E}[W_fW_h]=(f,h)_{L^2(\T^d)}.
	$$
Moreover, for any real-valued functions $f,h\in L^2(\T^d)$ with $\|f\|_{L^2}=\|h\|_{L^2}=1$, 
	\begin{align}\label{esperance-scalaire}
		\mathbb{E}[H_k(W_f)H_m(W_h)]=\delta_{km} k! [( f,h)_{L^2}]^k.
	\end{align} 
	We refer \cite{OhTho} for a proof. To represent the Wick-product as white noise functional, we denote
	$$ \eta_{N}(x,y):=\frac{1}{\widetilde{\sigma}_N}\sum_{|k|\leq N}\frac{1}{\sqrt{1+|k|^{2\alpha}} }\mathrm{e}^{ik\cdot(x+y)}.
	$$
	Then for
	$$ \phi_N^{\omega}(x):=\sum_{|k|\leq N}\frac{g_k(\omega)}{\sqrt{1+|k|^{2\alpha}}}\mathrm{e}^{ik\cdot x},
	$$
	we have 
	\begin{align}\label{whitenoiserepre}
	\phi_N(x)=\widetilde{\sigma}_NW_{\eta_N(x,\cdot)},\quad	 
	\phi_N^{\diamond l}(x)=H_l(\phi_N(x);\widetilde{\sigma}_N^2)=\widetilde{\sigma}_N^lH_l\big(W_{\eta_N(x,\cdot)}\big)
	.
	\end{align}

	Next we recall the Wiener chaos estimate. Let $(h_n)_{n\in\N}$ be a sequence of independent standard Gaussian random variables on a probability space $(\Omega,\mathcal{F},\mathbb{P})$. Given $k\in\N$ (including the $0$), we define the space of homogeneous Wiener chaos of degree $k$, $\mathcal{H}_k$, to be the closure in $L^2(\Omega,\mathbb{P})$ of polynomials $\prod_{n=1}^{\infty}H_{k_n}(g_n)$, where $\sum_{n=1}k_n=k$. Then we have the Ito-Wiener decomposition
	$$ L^2(\Omega,\mathbb{P})=\bigoplus_{k=0}^{\infty}\mathcal{H}_k.
	$$
	By the hypercontractivity, we have the following Wiener chaos estimate:
	\begin{prop}\label{Wienerchaos}
		Assume that $X\in \bigoplus_{j\leq k}\mathcal{H}_j$, then for any finite $p\geq 2$,
		$$ \|X\|_{L^p(\Omega)}\leq (p-1)^{\frac{k}{2}}\|X\|_{L^2(\Omega)}.
		$$ 
	\end{prop}

	
	\section{Convergence of the Gibbs measure}
	\label{sec:measure_convergence}

	\subsection{A variational formula for the partition function}
	
	The main strategy to prove Theorem~\ref{th:main_measure} and Proposition~\ref{pr:measure_counter_eg} is the recently developed variational approach to QFT (\cite{variational_QFT}). We first give a variational formula for $- \log \zZ_N$. We adapt the setting in \cite{quasi_inv_wave_3d}. 
	
	Let $\{B_{k}(\cdot)\}_{k \in \Z^2}$ be a collection of standard Brownian motions on the probability space $(\Omega, \fF, \P)$ such that $B_{k} = \overline{B_{-k}}$ and otherwise independent. Let
	\begin{equation*}
		X(t) = \sum_{k \in \Z^2} B_{k}(t) \mathrm{e}_k\;, 
	\end{equation*}
	which is the cylindrical Brownian motion on $L^2(\T^2)$ adapted to the filtration $(\fF_t)$ generated by $\{B_k\}$. 
	
	For every $N$, let $\sS_N$ be the operator such that
	\begin{equation} \label{eq:S_N}
		\widehat{\sS_N f}(k) = \frac{\widehat{f}(k)}{\scal{k}^{\alpha}} \cdot \one_{|k| \leq N}\;. 
	\end{equation}
	Let $W_N(t) := \sS_N X(t)$, and for every $N$, define the measure $\qQ_N$ by
	\begin{equation*}
		\frac{{\rm d}\Q_N}{{\rm d}\P} := \frac{1}{\zZ_N} e^{-\int_{\T^2} \widetilde{V_N}(W_N(1)){\rm d}x}\;. 
	\end{equation*}
	Here, the integration variable in $x$ is from $W_{N}(1) = W_{N}(1,\cdot)$. For $t=1$, we also simply write $W_N$ for $W_{N}(1)$. Then
	\begin{equation*}
		\Law_{\P}\big( W_N(1) \big) = \mu\;, 
	\end{equation*}
	and the normalisation constant $\zZ_N$ is the same as above. 
	
	By the martingale representation theorem, there exists an adapted $L^2$ process $u$ such that
	\begin{equation} \label{eq:RN_adapted_process}
		\frac{1}{\zZ_N} e^{-\int_{\T^2} \widetilde{V_N}(W_N) {\rm d}x} = \frac{{\rm d}\Q_N}{{\rm d}\P} = e^{\int_{0}^{1} \scal{u(t), {\rm d}X(t)} - \frac{1}{2} \int_{0}^{1} \|u(t)\|_{L^2}^2 {\rm d}t}\;. 
	\end{equation}
	Re-arranging the terms and taking logarithm, we get
	\begin{equation*}
		- \log \zZ_N = \int_{\T^2} \widetilde{V_{N}}(W_N) {\rm d}x + \int_{0}^{1} \scal{u(t), {\rm d}X(t)} - \frac{1}{2} \int_{0}^{1} \|u(t)\|_{L^2}^{2} {\rm d}t\;, 
	\end{equation*}
	where we recall the notation $W_N = W_N(1)$. Now, for the above $u$, define
	\begin{equation*}
		\widetilde{X}(t) := X(t) - \int_{0}^{t} u(s) {\rm d}s\;. 
	\end{equation*}
	Then by Girsanov theorem, $\widetilde{X}$ is a $\Q_N$ Brownian motion. Writing
	\begin{equation} \label{eq:I_N}
		\widetilde{W}_N(t):= \sS_N \widetilde{X}(t)\;, \quad \iI_N(v) = \sS_N \int_{0}^{1} v(s) {\rm d}s\;, 
	\end{equation}
	we get
	\begin{equation*}
		- \log \zZ_N = \int_{\T^2} V_{N} \big( \widetilde{W}_N + \iI_N(u) \big) {\rm d}x + \int_{0}^{1} \scal{u(t), {\rm d}\widetilde{X}(t)} + \frac{1}{2} \int_{0}^{1} \|u(t)\|_{L^2}^2 {\rm d}t\;. 
	\end{equation*}
	Note that the second term on the right hand side above is a martingale under $\Q_N$, and hence vanishes under $\E^{\Q_N}$. We have thus arrived at the following proposition. 
	
	\begin{prop}
		Let $u$ be the adapted $L^2$ process in \eqref{eq:RN_adapted_process}. Then we have the identity
		\begin{equation*}
			- \log \zZ_N = \E^{\Q_N} \Big[ \int_{\T^2} V_N \big( \widetilde{W}_N + \iI_N(u) \big) {\rm d}x + \frac{1}{2} \int_{0}^{1} \|u(t)\|_{L^2}^2 {\rm d}t \Big]\;, 
		\end{equation*}
		where $\widetilde{W}_N = \widetilde{W}_N(1) = \Pi_N \widetilde{X}(1)$, and $\Law_{\qQ_N} \big( \widetilde{W}_N \big) = \Law_{\mu}(\phi_N)$. 
	\end{prop}
	
	The above representation is sufficient for us to prove Theorem~\ref{th:main_measure}. But it will be convenient for us to be able to change the ``drift" $u$ freely while keeping the underlying probability space unchanged. For this reason, we use the following deeper variational formula. 
	
	\begin{prop} [\cite{variational_BP_formula, variational_Ustunel}]
		\label{pr:variational_BD}
		We have
		\begin{equation} \label{eq:variational}
			- \log \zZ_N = \inf_{v \in \HH_a} \E^{\P} \Big[ \int_{\T^2} \widetilde{V_{N}}(W_N + \mathcal{I}_N(v)) {\rm d}x + \frac{1}{2} \int_{0}^{1} \|v(t)\|_{L^2}^2 {\rm d}t \Big]\;, 
		\end{equation}
		where the infimum is taken over all predictable processes in $L^2$ with respect to the filtration generated by $X$. 
	\end{prop}
	
	Before we get into the proof of the main theorem, we first give a preliminary lemma controlling $\|\iI_N(v)\|_{H^\alpha}$ by the space-time $L^2$-norm of $v$. 
	
	\begin{lem} \label{le:Cameron_Martin}
		There exists $C>0$ such that
		\begin{equation*}
			\sup_{N} \|\iI_N(f)\|_{H^{\alpha}}^{2} \leq C \int_{0}^{1} \|f(t)\|_{L^2}^2 {\rm d}t
		\end{equation*}
		for all $f \in L^{2}\big( [0,1]; L^2(\T^2) \big)$. 
	\end{lem}
	\begin{proof}
		By definition of $\iI_N$, we have
		\begin{equation*}
			\widehat{\iI_N f}(k) = \mathbf{1}_{|k|\leq N}\frac{1}{\scal{k}^{\alpha}} \int_{0}^{1} \widehat{f}(t,k) {\rm d}t\;, 
		\end{equation*}
		and hence
		\begin{equation*}
			|\widehat{\iI_N f}(k)|^{2} \leq \frac{1}{\scal{k}^{2\alpha}} \Big| \int_{0}^{1} \widehat{f}(t,k) {\rm d}t \Big|^{2} \leq \scal{k}^{-2\alpha} \int_{0}^{1} |\widehat{f}(t,k)|^{2} {\rm d}t\;. 
		\end{equation*}
		As a consequence, we have
		\begin{equation*}
			\|\iI_N f\|_{H^\alpha}^{2} = \sum_{k} \scal{k}^{2\alpha} |\widehat{\iI_N f}(k)|^{2} \leq \sum_{k} \int_{0}^{1} |\widehat{f}(t,k)|^{2} {\rm d}t = C \int_{0}^{1} \|f(t)\|_{L^2}^2 {\rm d}t\;. 
		\end{equation*}
		The proof is complete. 
	\end{proof}

	\subsection{Necessity of the positivity condition -- proof of Proposition~\ref{pr:measure_counter_eg}}
	\label{sec:necessity_positivity}
	
	Suppose $V$ and $\rho$ are such that
	\begin{equation*}
		\sum_{j=2}^{m} \bar{a}_{j} \theta^{2j-4} < 0
	\end{equation*}
	for some $\theta \in \R$. By continuity, we can assume $\theta \neq 0$. Let $u = \theta N^{1-\alpha}$ which is certainly adapted. Write $U_N := \iI_N(u) \equiv \theta N^{1-\alpha}$. By Proposition~\ref{pr:variational_BD}, we have
	\begin{equation*}
		- \log \zZ_N \leq \E^{\P} \Big[ \int_{\T^2} \widetilde{V_{N}}(W_N + U_N) {\rm d}x + \frac{1}{2} \int_{0}^{1} \|u(t)\|_{L^2}^2 {\rm d}t \Big]\;. 
	\end{equation*}
	We will show that for the above drift $u$, the right hand side above is smaller than $-c N^{4(1-\alpha)}$ for some $c>0$. 
	
	For the term $\widetilde{V_{N}}(W_N + U_N)$, we have
	\begin{equation*}
		\widetilde{V_{N}}(W_N + U_N) = \sum_{j=2}^{m} \overline{a}_{j,N} N^{-(2j-4)(1-\alpha)} \; (W_N + U_N)^{\diamond (2j)} - \frac{1}{2} (W_N + U_N)^{\diamond 2} \;. 
	\end{equation*}
	Expanding the Wick product for each $j$ and re-organising the sum according to the power of $U_N$, we get
	\begin{equation} \label{eq:expansion}
		\widetilde{V_{N}}(W_N + U_N) = \sum_{\ell=0}^{2m} \sum_{j=2 \vee \frac{\ell}{2}}^{m} \overline{a}_{j,N}
		\begin{psmallmatrix}
			2j \\ \ell
		\end{psmallmatrix}
		N^{-(2j-4)(1-\alpha)} \; W_{N}^{\diamond (2j-\ell)} U_N^{\ell} - \frac{1}{2} \big( W_N^{\diamond 2} + 2 W_N U_N + U_N^2\big) \;, 
	\end{equation}
	where $U_{N}^{\ell}$ is the $\ell$-th power of $U_N = \iI_N(u)$, and $\diamond$ denotes the Wick product of $W_N$ with respect to its own Gaussian structure. 
	
	Note that for the terms in the above sum, the pointwise expectation $\E^{\P}$ is non-zero only when $\ell = 2j$. So for this drift $u$, we have
	\begin{equation*}
		\begin{split}
			\E \int_{\T^2} \widetilde{V_N} (W_N + \theta N^{1-\alpha}) {\rm d}x &= 4 \pi^2 \sum_{j=2}^{m} \overline{a}_{j,N} N^{-(2j-4)(1-\alpha)} (\theta N^{1-\alpha})^{2j} + \oO(N^{2(1-\alpha)})\\
			&< -c N^{4(1-\alpha)}
		\end{split}
	\end{equation*}
	for some $c>0$ (since $\theta \neq 0$). For the other term, we have
	\begin{equation*}
		\int_{0}^{1} \|u(t)\|_{L^2}^{2} {\rm d}t = C N^{2(1-\alpha)}\;. 
	\end{equation*}
	Hence, by the variational formula, we have the bound
	\begin{equation*}
		- \log \zZ_N < - c N^{4(1-\alpha)}
	\end{equation*}
	for all $N$, which implies that the densities
	\begin{equation*}
		\frac{1}{\zZ_N} e^{-\int_{\T^2} V_N(\Pi_N \phi)} \rightarrow 0
	\end{equation*}
    in probability with respect to $\mu$. But since their $L^1(\mu)$ norm are $1$, so it cannot converge in $L^1$. This completes the proof of Proposition~\ref{pr:measure_counter_eg}.

	\subsection{Proof of Theorem~\ref{th:main_measure}}
	
	\subsubsection{The main proposition and upper bound}
	
	Recall the renormalised potential $V_N$ and definition of the coefficients $\overline{a}_{j,N}$ in \eqref{eq:potential_renormalised_wick} and \eqref{eq:potential_coefficients_wick}. Let
	\begin{equation*}
		C_{N}^{(1)} = \overline{a}_{1,N} N^{2(1-\alpha)}\;, \qquad C_{N}^{(2)} = \big( \overline{a}_{0,N} - \overline{a}_{1,N} \sigma_N^2 \big) N^{4(1-\alpha)}. 
	\end{equation*}
	Writing $\phi_N = \Pi_N \phi$ for simplicity, we have
	\begin{equation*}
		V_{N}(\phi_N) = \sum_{j=2}^{m} \overline{a}_{j,N} N^{-(2j-4)(1-\alpha)} \phi_N^{\diamond (2j)}\;, 
	\end{equation*}
	where the Wick product is with respect to the Gaussian structure induced by $\mu$. In other words, we remove the $0$-th and $2$-nd chaos components from the polynomial. By standard hyper-contractivity arguments, one has
	\begin{equation*}
		\E^{\mu} \Big| \int_{\T^2} V_{N}(\phi_N) {\rm d}x - \lambda \int_{\T^2} \phi^{\diamond 4} {\rm d}x \Big|^{p} \rightarrow 0
	\end{equation*}
	as $N \rightarrow +\infty$. The key ingredient to pass the convergence to the level of exponential is the following uniform bound. 
	
	\begin{prop} \label{pr:uniform_bd_density}
		For every $p \geq 1$, we have
		\begin{equation*}
			\sup_{N \in \N} \E^{\mu} \Big( e^{- p \int_{\T^2} \widetilde{V_{N}}(\phi_N) {\rm d}x} \Big) < +\infty\;, \quad \E^{\mu} \Big( e^{- \lambda p \int_{\T^2} \phi^{\diamond 4} {\rm d}x + \frac{p}{2} \int_{\T^2} \phi^{\diamond 2} {\rm d}x } \Big) < +\infty. 
		\end{equation*}
	\end{prop}
	
	We first show how Theorem~\ref{th:main_measure} follows from Proposition~\ref{pr:uniform_bd_density}. 
	
	\begin{proof}
		[Proof of Theorem~\ref{th:main_measure}] Note that under $\mu$, we have
		\begin{equation*}
			\widetilde{V_N}(\phi_N) = V_N(\phi_N) - \frac{1}{2} \phi_N^{\diamond 2}\;,
		\end{equation*}
	    so it suffices to prove the corresponding statement with $V_N$ instead of $\widetilde{V_N}$ and with the $\phi^{\diamond 2}$ removed in the limiting measure. Since
		\begin{equation*}
			\int_{\T^2} V_N(\phi_N) {\rm d}x \rightarrow \lambda \int_{\T^2} \phi^{\diamond 4} {\rm d}x
		\end{equation*}
		in probability, and since the exponential function is continuous, we have
		\begin{equation*}
			e^{- \int_{\T^2} V_{N}(\phi_N) {\rm d}x} \rightarrow e^{- \lambda \int_{\T^2} \phi^{\diamond 4} {\rm d}x}
		\end{equation*}
		in probability as well. Theorem~\ref{th:main_measure} then follows from the convergence in probability together with the uniform bounds in Proposition~\ref{pr:uniform_bd_density} (with a larger $p$). 
	\end{proof}

	We now turn to proving Proposition~\ref{pr:uniform_bd_density}. We only need to prove the first bound, as the second one is the special case with $\overline{a}_{2,N} = \lambda > 0$ and $\overline{a}_{j,N}=0$ for all other $j$. Also, by replacing $\overline{a}_{j,N}$ with $p \overline{a}_{j,N}$, the assumption \eqref{eq:average_V_positive} is not affected. Hence we can assume without loss of generality that $p=1$. 
	
	It suffices to prove a uniform-in-$N$ bound for $|\log \zZ_N|$. Jensen's inequality gives
	\begin{equation*}
		- \log \zZ_N = - \log \E^{\mu} \Big[ e^{-\int_{\T^2} \widetilde{V_{N}}(\phi_N) {\rm d}x} \Big] \leq \E^{\mu} \Big[ \int_{\T^2} \widetilde{V_{N}}(\phi_N) {\rm d}x \Big] = 0\;. 
	\end{equation*}
	So it remains to prove a lower bound for $- \log \zZ_N$. The rest of the section will be devoted to that.

	\subsubsection{Expansion}
	
	By the variational formula \eqref{eq:variational}, it suffices to prove a lower bound of its right hand side uniform over $N$ and all $L^2$ adapted process $u$. 
	
	Starting from the expansion \eqref{eq:expansion} and re-organising the sums, we have
	\begin{equation} \label{eq:V_N_organised}
		V_{N}(W_N + U_N) = \sum_{\ell=0}^{2m-1} \yY_{N,\ell} U_{N}^{\ell} - \frac{1}{2} W_N^{\diamond 2} - W_N U_N + \sum_{j=2}^{m} \overline{a}_{j,N} N^{-(2j-4)(1-\alpha)}U_{N}^{2j} - U_N^2 \;, 
	\end{equation}
	where
	\begin{equation*}
		\yY_{N,\ell} = \sum_{j = 2 \vee (\lfl \frac{\ell}{2} \rfl + 1)}^{m} \overline{a}_{j,N}
		\begin{psmallmatrix}
			2j \\ \ell
		\end{psmallmatrix}
		N^{-(2j-4)(1-\alpha)} \; W_{N}^{\diamond (2j-\ell)}\;, 
	\end{equation*}
	and we have separated out the terms with $2j=\ell$ in the sum. Note that the sum in $\ell$ (in the first term) is up to $2m-1$ since the last one ($\ell=2m$) is separated into the second term in \eqref{eq:V_N_organised}, so the sum defining $\yY_{N,\ell}$ is empty when $\ell = 2m$.

	\begin{prop}\label{coercivity}
		If the positivity condition \eqref{eq:average_V_positive} holds, then there exists $c, C>0$ such that
		\begin{equation*}
			\sum_{j=2}^{m} \overline{a}_{j,N} N^{-(2j-4)(1-\alpha)} U_{N}^{2j} - U_N^2 \geq c \big( U_{N}^{4} + N^{-(2m-4)(1-\alpha)} U_{N}^{2m} \big) - C
		\end{equation*}
		for all sufficiently large $N$. As a consequence, we have
		\begin{equation} \label{eq:V_N_lower_expansion}
			\begin{split}
				&\phantom{11}\int_{\T^2} \widetilde{V_N}(W_N + U_N) {\rm d}x + \frac{1}{2} \int_{0}^{1} \|u(t)\|_{L^2}^{2} {\rm d}t\\
				&\geq  \int_{\T^2} \Big( \sum_{\ell=0}^{2m-1} \yY_{N,\ell} U_{N}^{\ell} - \frac{1}{2} W_N^{\diamond 2} - W_N U_N \Big) {\rm d}x - C \\
				&\phantom{11}+ c \Big( \|U_N\|_{L^4}^{4} + N^{-(2m-4)(1-\alpha)} \|U_N\|_{L^{2m}}^{2m} + \|U_N\|_{H^{\alpha}}^{2} \Big)\;, 
			\end{split}
		\end{equation}
		where $U_N = \iI_N(u)$, and $\iI_N$ is defined in \eqref{eq:I_N}.  
	\end{prop}
	\begin{proof}
		The first claim follows from the positivity assumption \eqref{eq:average_V_positive}, the convergence $\overline{a}_{j,N} \rightarrow \overline{a}_j$ for every $j$, and that
		\begin{equation*}
			U_N^2 \leq M + \frac{U_N^4}{M}
		\end{equation*}
	    for every $M \geq 1$. The second claim is a consequence of the first one and Lemma~\ref{le:Cameron_Martin}. 
	\end{proof}

	Our next aim is to show that for every sufficiently small $\delta > 0$, there exists constant $C=C(\delta,m)$ such that
	\begin{equation} \label{eq:absorb_aim}
		\begin{split}
			&\phantom{111} \Big| \int_{\T^2} W_{N}^{\diamond 2} {\rm d}x \Big| + \Big| \int_{\T^2} W_N U_N {\rm d}x \Big| + \sum_{\ell=0}^{2m-1} \Big| \int_{\T^2} \yY_{N,\ell} U_{N}^{\ell} {\rm d}x \Big|\\
			&\leq C \qQ_N(W_N) + \delta \Big( \|U_N\|_{L^4}^{4} + N^{-(2m-4)(1-\alpha)} \|U_N\|_{L^{2m}}^{2m} + \|U_N\|_{H^{\alpha}}^{2} \Big)\;,
		\end{split}
	\end{equation}
	where $\qQ_N(W_N)$ is some function depending on suitable (negative) Sobolev norm of $W_N$ whose expectation is uniformly bounded in $N$. If \eqref{eq:absorb_aim} is true, then we can combine it with \eqref{eq:V_N_lower_expansion} and Proposition~\ref{pr:variational_BD} to conclude the lower bound
	\begin{equation*}
		- \log \zZ_N \geq - C \E \big(  \qQ(W_N) \big) > - C
	\end{equation*}
	for some $C$ independent of $N$. Also note that it suffices to show that each term on the left hand side satisfies the bound. The rest of this section is devoted to the proof of \eqref{eq:absorb_aim}. 
	
	\subsubsection{The first two terms}
	The bounds for the first two terms on the left hand side of \eqref{eq:absorb_aim} are straightforward. For the first one, we have
	\begin{equation*}
		\Big| \int_{\T^2} W_N^{\diamond 2} {\rm d}x \Big| \leq \|W_{N}^{\diamond 2}\|_{H^{-2(1-\alpha)-\eps}}\;,
	\end{equation*}
    which is of the form $\qQ_N(W_N)$. For the second one, we have
    \begin{equation*}
    	\Big| \int_{\T^2} W_N U_N {\rm d}x \Big| \leq \|W_N\|_{H^{-\alpha}} \|U_N\|_{H^{\alpha}} \leq \frac{1}{\delta} \|W_N\|_{H^{-\alpha}}^2 + \delta \|U_N\|_{H^\alpha}^2\;,
    \end{equation*}
	which is again of the desired form. 
	
	\subsubsection{The case $0 \leq \ell \leq 3$}
	We now turn to the terms $\yY_{N,\ell} U_N^{\ell}$. We first consider the case when $0 \leq \ell \leq 3$. A typical term in $\yY_{N,\ell}$ for $0 \leq \ell \leq 3$ is of the form $$N^{-(2j-4)(1-\alpha)} \scal{W_{N}^{\diamond (2j-\ell)}, U_{N}^{\ell}}$$ for $j = 2, \dots, m$, where $\scal{\cdot,\cdot}$ denotes the $L^2(\T^2)$ inner product. 
	
	The term $\ell=0$ corresponds to $N^{-(2j-4)(1-\alpha)} \int W_{N}^{\diamond (2j)}$. It satisfies the bound
	\begin{equation*}
		N^{-(2j-4)(1-\alpha)} \Big| \int W_{N}^{\diamond (2j)} \Big| \lesssim N^{-(2j-4)(1-\alpha)} \|W_{N}^{\diamond (2j)}\|_{H^{-\beta}}
	\end{equation*}
	for every $\beta \geq 0$. By Lemma~\ref{le:Wick_bd}, since $\alpha \in (\frac{3}{4},1)$, its expectation is uniformly bounded in $N$ as long as $\beta>1$. Hence, we can take  $\qQ_N(W_N) = N^{-(2j-4)(1-\alpha)} \|W_{N}^{\diamond (2j)}\|_{H^{-(1+\eps)}}$ which satisfies the requirements for the bound \eqref{eq:absorb_aim}. 
	
	For $\ell=1$, it follows from duality and Cauchy-Schwarz that
	\begin{equation*}
		\begin{split}
			\phantom{111}N^{-(2j-4)(1-\alpha)} \big| \scal{W_{N}^{\diamond (2j-1)}, U_{N}} \big|
			&\leq N^{-(2j-4)(1-\alpha)} \|W_N^{\diamond (2j-1)}\|_{H^{-\alpha}} \|U_{N}\|_{H^\alpha}\\
			&\leq \delta^{-1} N^{-(4j-8)(1-\alpha)} \|W_N^{\diamond (2j-1)}\|_{H^{-\alpha}}^{2} + \delta \|U_N\|_{H^\alpha}^{2}\;. 
		\end{split}
	\end{equation*}
	By Lemma~\ref{le:Wick_bd}, the quantity $N^{-(4j-8)(1-\alpha)} \E \|W_{N}^{\diamond (2j-1)}\|_{H^{-\alpha}}^{2}$ is uniformly bounded in $N$ as long as $1 \wedge \alpha > 3(1-\alpha)$, which is the case for $\alpha \in (\frac{3}{4},1)$. So the desired bound \eqref{eq:absorb_aim} is true for $\ell=1$. 
	
	For $\ell = 2$, let $\beta>0$ to be specified later, and $p, q, q_1, q_2 \in (1,+\infty)$ be such that
	\begin{equation*}
		\frac{1}{p} + \frac{1}{q} = 1 \quad \text{and} \quad \frac{1}{q_1} + \frac{1}{q_2} = \frac{1}{q}\;. 
	\end{equation*}
	By duality of $W^{-\beta,p}$ and $W^{\beta,q}$ and then Lemma~\ref{le:fractional_Leibniz}, we have
	\begin{equation*}
		\begin{split}
			\big| \scal{W_{N}^{\diamond (2j-2)}, U_{N}^2} \big| &\leq \|W_{N}^{\diamond (2j-2)}\|_{W^{-\beta,p}} \|U_N^2\|_{W^{\beta,q}}\\
			&\lesssim \|W_{N}^{\diamond (2j-2)}\|_{W^{-\beta,p}} \|U_N\|_{W^{\beta,q_1}} \|U_N\|_{L^{q_2}}\;. 
		\end{split}
	\end{equation*}
	We furthermore choose $q$ sufficiently close to $1$ so that $q_1 \leq 2$ and $q_2 \leq 4$, and choose $\beta \in ( 2(1-\alpha),\alpha)$. This is possible as long as $\alpha > \frac{2}{3}$, which is in the range of our assumption. 
	
	Then multiplying both sides by $N^{-(2j-4)(1-\alpha)}$ and using H\"older to split the three terms, we get
	\begin{equation*}
		\begin{split}
			&\phantom{11}N^{-(2j-4)(1-\alpha)} \big| \scal{W_{N}^{\diamond (2j-2)} U_{N}^2} \big|\\
			&\lesssim N^{-(2j-4)(1-\alpha)} \|W_{N}^{\diamond (2j-2)}\|_{W^{-\beta,p}} \|U_N\|_{H^\alpha} \|U_{N}\|_{L^4}\\
			&\lesssim \delta^{-3} N^{-4(2j-4)(1-\alpha)} \|W_{N}^{\diamond (2j-2)}\|_{W^{-\beta,p}}^{4} + \delta \big( \|U_N\|_{H^\alpha}^2 + \|U_N\|_{L^4}^4 \big)\;, 
		\end{split}
	\end{equation*}
	where the proportionality constant does not depend on $\delta$. By Lemma~\ref{le:Wick_bd}, the first term above has finite (uniform-in-$N$) expectation since $\beta > 2(1-\alpha)$. Hence, it is of the form of the right hand side of \eqref{eq:absorb_aim}. The completes the case $\ell=2$. 
	
	For $\ell=3$, by duality and Lemma~\ref{le:fractional_Leibniz}, we have
	\begin{equation*}
		\begin{split}
			\big| \scal{W_{N}^{\diamond (2j-3)}, U_N^3} \big| &\leq \|W_{N}^{\diamond (2j-3)}\|_{W^{-\beta,\frac{1+\eps}{\eps}}} \|U_N^3\|_{W^{\beta,1+\eps}}\\
			&\lesssim \|W_{N}^{\diamond (2j-3)}\|_{W^{-\beta,\frac{1+\eps}{\eps}}} \|U_N\|_{W^{\beta,\frac{2(1+\eps)}{1-\eps}}} \|U_N\|_{L^4}^2\;, 
		\end{split}
	\end{equation*}
	where $\eps, \beta > 0$ are to be specified later. By Proposition~\ref{pr:general_Sobolev}, for $\beta < \alpha$, we have
	\begin{equation*}
		\|U_N\|_{W^{\beta,\frac{2(1+\eps)}{1-\eps}}} \lesssim_\eps \|U_N\|_{W^{\alpha,p}}^{\frac{\beta}{\alpha}} \|U_N\|_{L^{4}}^{1-\frac{\beta}{\alpha}} \lesssim \|U_N\|_{H^{\alpha}}^{\frac{\beta}{\alpha}} \|U_N\|_{L^4}^{1-\frac{\beta}{\alpha}}\;, 
	\end{equation*}
	where
	\begin{equation*}
		p = \frac{4(1+\eps)\beta}{(1-3\eps)\alpha + (1+\eps)\beta} < 2
	\end{equation*}
	if $\beta<\alpha$ and $\eps$ is sufficiently small (depending on $\alpha$, $\beta$), and hence the second inequality above (relaxing $W^{\alpha,p}$ to $H^\alpha$) is valid. Plugging it back into the original term and applying H\"older, we get
	\begin{equation*}
		\begin{split}
			&\phantom{11}N^{-(2j-4)(1-\alpha)} \big| \scal{W_{N}^{\diamond (2j-3)}, U_N} \big|\\
			&\lesssim_\eps C_{\delta} \Big( N^{-(2j-4)(1-\alpha)} \|W_{N}^{\diamond (2j-3)}\|_{W^{-\beta,\frac{1+\eps}{\eps}}} \Big)^{\frac{4\alpha}{\alpha-\beta}} + \delta \big( \|U_N\|_{H^\alpha}^2 + \|U_N\|_{L^4}^{4} \big)\;. 
		\end{split}
	\end{equation*}
	Again by Lemma~\ref{le:Wick_bd}, if we choose $\beta > 1-\alpha$, then the expectation of the first term above will be uniformly bounded in $N$, and hence satisfies the form of \eqref{eq:absorb_aim}. Recall that we have also required $\beta<\alpha$ when applying Proposition~\ref{pr:general_Sobolev} in the previous step. This is possible if $1-\alpha < \alpha$, which is true as long as $\alpha > \frac{1}{2}$ (which satisfies our assumption $\alpha \in (\frac{3}{4},1)$). 
	
	We have thus established the desired bound for $0 \leq \ell \leq 3$.

	\subsubsection{The case $4 \leq \ell \leq 2m-1$}
	
	We now turn to the situation when $4 \leq \ell \leq 2m-1$. The relevant terms to control here are $N^{-(2j-4)(1-\alpha)} \scal{W_{N}^{\diamond (2j-\ell)}, U_{N}^{\ell}}$ where $4 \leq \ell \leq 2m-1$ and $2j-\ell \geq 1$. We will prove the following proposition. 
	
	\begin{prop}
		Fix $4 \leq \ell \leq 2m-1$ and $j \leq m$ such that $2j-\ell \geq 1$. Let $m_0 = \lfl \frac{\ell}{2} \rfl + 1$. Then $m_0 \leq m$, and for every $\delta > 0$, there exists $C_\delta$ such that
		\begin{equation*}
			\begin{split}
				&\phantom{111}N^{-(2j-4)(1-\alpha)} \big| \scal{W_{N}^{\diamond (2j-\ell)}, U_{N}^{\ell}} \big|\\
				&\leq C_{\delta} \qQ_N(W_N) + \delta \big( \|U_N\|_{H^{\alpha}}^{2} + N^{-(2m_0-4)(1-\alpha)} \|U_N\|_{L^{2 m_0}}^{2m_0} \big)\;, 
			\end{split}
		\end{equation*}
		where $\qQ_N (W_N)$ is a positive function depending on certain negative Sobolev norm of $W_N$, and its expectation is uniformly bounded in $N$. The constant $C_{\delta}$ is independent of $N$. 
	\end{prop}
	\begin{proof}
		We divide the argument into several steps. 
		\begin{flushleft}
			\emph{Step 1}:
		\end{flushleft}
		Let $\beta,\eps>0$ be two parameters whose values will be specified later. By duality and repeated applications of Lemma~\ref{le:fractional_Leibniz}, we have
		\begin{equation*}
			\begin{split}
				\big| \scal{W_{N}^{\diamond (2j-\ell)}, U_N^{\ell}} \big| &\leq \|W_N^{\diamond (2j-\ell)}\|_{W^{-\beta,\frac{1+\eps}{\eps}}} \|U_N^{\ell}\|_{W^{\beta,1+\eps}}\\
				&\lesssim_\eps \|W_{N}^{\diamond (2j-\ell)}\|_{W^{-\beta,\frac{1+\eps}{\eps}}} \|U_N\|_{W^{\beta,p_\eps}} \|U_N\|_{L^{2m_0}}^{\ell-1}\;, 
			\end{split}
		\end{equation*}
		where
		$
		p_\eps = \frac{2(1+\eps) m_0}{2m_0 - (1+\eps)(\ell-1)}\;,
		$
		and it decreases to $\frac{2 m_0}{2 m_0 - \ell+1}$ as $\eps \rightarrow 0$. If $\beta < \alpha$, then by Proposition~\ref{pr:general_Sobolev}, we can further control the quantity $\|U_N\|_{W^{\beta,p_\eps}}$ by
		\begin{equation*}
			\|U_N\|_{W^{\beta,p_\eps}} \lesssim \|U_N\|_{W^{\alpha,q_\eps}}^{\frac{\beta}{\alpha}} \|U_N\|_{L^{2m_0}}^{1-\frac{\beta}{\alpha}}\;, 
		\end{equation*}
		where $q_\eps = \frac{2m_0 \beta}{\frac{2 m_0 \alpha}{p_\eps} - (\alpha-\beta)}$,  and $q_\eps$ decreases to $\frac{2 m_0 \beta}{(2m_0-\ell)\alpha+\beta}$ as $\eps \rightarrow 0$. Hence, if we choose $\beta$ such that
		\begin{equation} \label{eq:beta_constraint_1}
			\frac{2 m_0 \beta}{(2m_0-\ell)\alpha + \beta} < 2\;, 
		\end{equation}
		and choose $\eps>0$ sufficiently small (depending on $\beta$), then $q_\eps < 2$ and we can relax $\|U_N\|_{W^{\alpha,q_\eps}}$ to $\|U_N\|_{H^\alpha}$. Also relaxing $\|W_{N}^{\diamond (2j-\ell)}\|_{W^{-\beta,\frac{1+\eps}{\eps}}}$ to $\|W_{N}^{\diamond (2j-\ell)}\|_{\cC^{-\beta}}$, we obtain the bound
		\begin{equation} \label{eq:bd_intermediate_1}
			|\scal{W_{N}^{\diamond (2j-\ell)}, U_N^{\ell}}| \lesssim \|W_N^{\diamond (2j-\ell)}\|_{C^{-\beta}} \|U_N\|_{H^\alpha}^{\frac{\beta}{\alpha}} \|U_N\|_{L^{2m_0}}^{\ell-\frac{\beta}{\alpha}}\;. 
		\end{equation}
		The proportionality constant depends on the parameters $\alpha$ and $\beta$ but is independent of $N$. Note that the right hand side as well as the proportionality constant does not depend on $\eps$. 
		
		Note that we have previously chosen $\beta < \alpha$. But with the assumption on $m_0$, this is implied by the constraint \eqref{eq:beta_constraint_1}. Hence the only constraint for \eqref{eq:bd_intermediate_1} to hold is \eqref{eq:beta_constraint_1}. 
		\begin{flushleft}
			\begin{flushleft}
				\textit{Step 2. }
			\end{flushleft}
		\end{flushleft}
		We re-write the bound \eqref{eq:bd_intermediate_1} as
		\begin{equation*}
			\begin{split}
				\phantom{111}N^{-(2j-4)(1-\alpha)} |\scal{W_{N}^{\diamond (2j-\ell)}, U_N^{\ell}}|
				&\lesssim \frac{\|W_{N}^{\diamond (2j-\ell)}\|_{C^{-\beta}}}{N^{(2j-4)(1-\alpha)-\gamma}} \cdot \|U_N\|_{H^{\alpha}}^{\frac{\beta}{\alpha}} \cdot \big( N^{-\frac{\gamma \alpha}{\ell \alpha-\beta}} \|U_N\|_{L^{2m_0}} \big)^{\ell-\frac{\beta}{\alpha}}\;. 
			\end{split}
		\end{equation*}
		Hence, if we choose $\gamma$ such that
		\begin{equation} \label{eq:gamma_choice}
			\gamma \cdot \frac{2m_0 \alpha}{\ell \alpha - \beta} = (2m_0-4)(1-\alpha) \Longleftrightarrow \gamma = \frac{(m_0-2)(\ell \alpha - \beta)(1-\alpha)}{m_0 \alpha}\;, 
		\end{equation}
		we can use H\"older to separate the three terms in the product above so that
		\begin{equation*}
			\begin{split}
				&\phantom{111}N^{-(2j-4)(1-\alpha)} |\scal{W_{N}^{\diamond (2j-\ell)}, U_N^{\ell}}| \\
				\leq &C_{\delta} \big\| \frac{W_{N}^{\diamond (2j-\ell)}}{N^{(2j-4)(1-\alpha)-\gamma}} \big\|_{\mathcal{C}^{-\beta}}^{\eta} + \delta \big( \|U_N\|_{H^\alpha}^{2} + N^{-(2m_0-4)(1-\alpha)} \|U_N\|_{L^{2m_0}}^{2m_0} \big)\;, 
			\end{split}
		\end{equation*}
		where
		\begin{equation*}
			\eta = \frac{2 m_0 \alpha}{(2m_0 - \ell)\alpha - (m_0-1)\beta}\;. 
		\end{equation*}
		Note that the use of H\"older and hence the above bound is valid if $\eta > 1$, which is implied by the constraint \eqref{eq:beta_constraint_1}. 
		
		\begin{flushleft}
			\textit{Step 3. }
		\end{flushleft}
		It then remains to show that for every $\alpha \in (\frac{3}{4}, 1)$, there exists $\beta$ satisfying \eqref{eq:beta_constraint_1} such that for $\gamma$ given in \eqref{eq:gamma_choice} and
		\begin{equation*}
			\qQ_N(W_N) = \frac{\|W_{N}^{\diamond (2j-\ell)}\|_{\cC^{-\beta}}}{N^{(2j-4)(1-\alpha)-\gamma}}\;, 
		\end{equation*}
		one has
		\begin{equation*}
			\sup_{N} \E |\qQ_N(W_N)|^{\eta} < +\infty. 
		\end{equation*}
		This is equivalent to the following two constraints on $(\beta,\gamma)$: 
		\begin{enumerate}
			\item $(2j-4)(1-\alpha) - \gamma \geq 0$; 
			\item $\beta \wedge 1 > \gamma - (\ell-4)(1-\alpha)$. 
		\end{enumerate}
		We first check the second one. Note that \eqref{eq:beta_constraint_1} implies $\beta < \alpha < 1$, so the left hand side $\beta \wedge 1$ could be replaced by $\beta$. Routine algebraic calculations then show that the second constraint above is equivalent to
		\begin{equation} \label{eq:beta_constraint_2}
			\beta > \frac{2 (2m_0 - \ell) \alpha (1-\alpha)}{m_0 + 2 \alpha - 2}\;. 
		\end{equation}
		Combing \eqref{eq:beta_constraint_1} and \eqref{eq:beta_constraint_2}, we see that a possible choice of $\beta$ exists if
		\begin{equation*}
			\frac{2 (2m_0 - \ell) \alpha (1-\alpha)}{m_0 + 2 \alpha - 2} < \frac{2m_0 - \ell}{m_0 - 1} \cdot \alpha\;, 
		\end{equation*}
		which is true as long as $\alpha > \frac{1}{2}$. 
		
		It remains to check the first constraint above. This can be reduced to
		\begin{equation} \label{eq:beta_constraint_3}
			\beta \geq \frac{(4m_0 - 2 \ell) - (2j-\ell) m_0}{m_0 - 2} \cdot \alpha\;. 
		\end{equation}
		Combing it with \eqref{eq:beta_constraint_1}, we see that a possible choice of $\beta$ exists if
		\begin{equation*}
			\frac{(4m_0 - 2 \ell) - (2j-\ell) m_0}{m_0 - 2} < \frac{2 m_0 - \ell}{m_0 - 1}\;, 
		\end{equation*}
		which holds if $2j-\ell \geq 1$ and $m_0 = \lfl \frac{\ell}{2} \rfl + 1 \leq \ell -1$. 
		
		We have thus shown that for $\alpha \in (\frac{3}{4}, 1)$, there exists choice of $\beta$ and $\gamma$ as specified above so that all the bounds hold. This completes the proof of the proposition. 
	\end{proof}

	\section{The wave dynamics}
	\label{sec:wave_convergence}

	Consider the wave dynamics:
	\begin{equation}\label{Universality-N}
		\begin{cases} 
			&\d_t^2 u_N + |\nabla|^{2\alpha} u_N + \Pi_N V_N'(\Pi_N u_N) = 0\;,\\
			& (u_N, \d_t u_N)|_{t=0} = \Pi_N\vec{\phi},
		\end{cases}
	\end{equation}
where
$$ \Pi_N\vec{\phi}:=\frac{1}{2\pi}\Big(\sum_{|k|\leq N}\frac{g_k(\omega)}{\sqrt{1+|k|^{2\alpha}}}\mathrm{e}^{ik\cdot x},\; 
\sum_{|k|\leq N}h_k(\omega)\mathrm{e}^{ik\cdot x}
\Big).
$$
	Denote by
	\begin{align*}  V_N(\varphi):=\ov{a}_{1,N}N^{2\beta}H_2(\varphi;\widetilde{\sigma}_N^2)+W_N(\varphi),\quad W_N(\varphi):=\sum_{j=2}^m\ov{a}_{j,N}N^{-(2j-4)\beta}H_{2j}(\varphi;\widetilde{\sigma}_N^2),
	\end{align*}
	where $\kappa_N:=2a_{1,N}N^{2\beta}-1$. We rewrite the equation \eqref{Universality-N} as
	\begin{equation}\label{Universality-N'}
		\begin{cases} 
		&\d_t^2 u_N + (\dD^{\alpha})^2 u_N +\kappa_N\Pi_Nu_N+ \Pi_N W_N'(\Pi_N u_N) = 0\;,\\ &(u_N, \d_t u_N)|_{t=0} = \Pi_N\vec{\phi}\;.
		\end{cases}
	\end{equation}
	Note that for each fixed $N$, \eqref{Universality-N'} is globally well-posed. Indeed, when writing in Fourier variables, the equation \eqref{Universality-N'} is a finite-dimensional system and its local well-posedness is ensured by the Cauchy-Lipschitz Theorem. Moreover, the conserved energy
	$$ \mathcal{E}(u_N(t)):=\int_{\T^2}\big(\frac{1}{2}(|\partial_tu_N|^2+||\nabla|^{\alpha}u_N|^2 )+V_N(\Pi_Nu_N)\big)  {\rm d}x
	$$ 
	is a Lyapunov functional that controls the quantity 
	$$ \|\partial_tu_N(t)\|_{L^2(\T^2)}^2+\|u_N\|_{H^{\alpha}(\T^2)}^2+N^{-(2m-4)\beta}\|u_N\|_{L^{2m}(\T^2)}^{2m}-C_{N,m}\|u_N\|_{L^2(\T^2)}^2.
	$$ 
	Since
	$$ \|u_N\|_{L^2(\T^2)}^2\leq C\|u_N\|_{L^{2m}(\T^2)}^{2}\leq \frac{1}{2C_{N,m}}N^{-(2m-4)\beta}\|u_N\|_{L^{2m}(\T^2)}^{2m}+C'_{N,m},
	$$
	we deduce that $u_N$ cannot blowup in finite time. We denote by $\vec{\Phi}_N(t)$ the flow of \eqref{Universality-N'}, and we recall that $\vec{\nu}_N$ is invariant under $\vec{\Phi}_N(t)$.

	In this section, we will prove Theorem \ref{thm:dynamics} with more precise statements: the well-posedness of the renormalized and the convergence of \eqref{Universality-N}.
	Heuristically, recall from Proposition \ref{convergeneclinear} that $\kappa_N\rightarrow \kappa\in\R$ and
	$$ |\kappa_N-\kappa|\leq CN^{-(2\alpha-1)}.
			$$
	Then formal analysis suggests that as $N\rightarrow\infty$, \eqref{Universality-N'} should converge to the renormalized cubic wave equation
	\begin{align}\label{cubic} \partial_t^2u+(\dD^{\alpha})^2u+\kappa u+4\ov{a}_2 u^{\diamond 3}=0,\quad (u,\partial_tu)|_{t=0}=\vec{\phi}
	\end{align}
	where $$u^{\diamond 3}:=\lim_{N\rightarrow\infty}\Pi_NH_4(\Pi_Nu;\widetilde{\sigma}_N^2)$$ is a well-defined object on the support of $\mu$.
	The goal of this section is to rigorously justify the above convergence. 

\subsection{More notations}	
	Before presenting the main propositions, we need more notations.
	Define the linear propagators $\sS(t)$ and $\sS'(t)$ by
			\begin{equation*}
				\sS(t) \vec{f} = \cos(t \dD^{\alpha}) f + \frac{\sin(t \dD^{\alpha})}{\dD^{\alpha}} f'\;, \quad \sS'(t) \vec{f} = - \scal{\dD}^{\alpha} \sin(t \dD^{\alpha}) f + \cos(t \dD^{\alpha}) f'\;, 
			\end{equation*}
			and let $\vec{\sS}(t) = \big(\sS(t), \sS'(t)\big)$, where we again  write $\vec{f} = (f,f')$. 
				For every $\vec{\phi} \in \dD'(\T^2) \times \dD'(\T^2)$, denote
			\begin{equation*}
				\<1>(t, \cdot) = \<1>(\vec{\phi})(t, \cdot):= \sS(t) \vec{\phi}\; \text{ and }\vec{\<1>}:=\vec{\mathcal{S}}(t)\vec{\phi}.
			\end{equation*}
			Sometimes we will omit the dependence on $\vec{\phi}$ in the notation $\<1>$ for simplicity. For an integer $N\in\N$, denote
			$ \<1>_N:=\Pi_N\<1>
			$ and $\vec{\<1>}_N:=\Pi_N\vec{\<1>}$.

We will frequently use two	mall parameters $\eps,\theta_0$ such that
	$$  \theta_0 \ll \eps\ll 1.
	$$			
Throughout this section, the symbol $\eps\ll 1$ always means that
$$  \eps<2^{-100m}\times \beta^{100}. 
$$
Recall that	$ \beta=1-\alpha,s_0=4\alpha-3.
$
	 Since the flow of the wave equation is vector-valued, we denote by
		$$ \mathcal{H}^s:=H^s\times H^{s-\alpha},\quad \mathcal{W}^{s,r}:=W^{s,r}\times W^{s-\alpha,r}.
		$$
 For given functions $f,\vec{f}=(f,f')$ and $I\subset\R$, we define for $\sigma\in\R$ the norms
$$ \|f\|_{Y^{\sigma}(I)}:=\|f\|_{L_t^{\infty}H^{\sigma}(I)}+\|f\|_{L_t^{2+\theta_0}L_x^{2+\frac{4}{\theta_0}}(I) }
$$
and
$$ \|\vec{f}\|_{\mathcal{Y}^{\sigma}(I)}:=\|\vec{f}\|_{L_t^{\infty}\mathcal{H}^{\sigma}(I)}+\|f\|_{L_t^{2+\theta_0}L_x^{2+\frac{4}{\theta_0}}(I) }.
$$
Note that the norm $\big(2+\theta_0,2+\frac{4}{\theta_0}
\big)$ is Strichartz admissible. For the solution $u$ of
$$ \partial_t^2u+(\dD^{\alpha})^2u=F, \quad (t,x)\in I\times \T^2,
$$
 we will use in particular the following inequality
\begin{align}\label{Strichartz2d} 
\|(u(t),\partial_tu(t))\|_{Y^{s_1}(I)}\lesssim \|(u(t_0),\partial_tu(t_0))\|_{\mathcal{H}^{s_1}}+\|F\|_{L_t^1H^{s_1-\alpha}(I)},	
\end{align}
provided that $s_1>\frac{2-\alpha}{2+\theta_0}$ and $t_0\in I$.

	
	\subsection{Well-posedness for the cubic equation}
	We sketch the almost sure global well-posedness of \eqref{cubic} whenever $\alpha>\frac{8}{9}$. The local well-posedness follows the recentering scheme of Bourgain \cite{Bo}, while the global well-posedness follows the invariant argument of Bourgain \cite{Bo}. 
	
	Consider the truncated equation
	\begin{align}\label{cubicN} \partial_t^2v_N+(\dD^{\alpha})^{2}v_N+\Pi_N(\kappa u_N+4\ov{a}_2u_N^{\diamond 3})=0,\quad (u_N,\partial_tu_N)|_{t=0}=\Pi_N\vec{\phi}.
	\end{align}
	Denote by
	$$ \mathcal{J}_{t_0}:=\int_{t_0}^t\frac{\sin((t-t')\dD^{\alpha} )}{\dD^{\alpha}}dt', 
	$$
	the Duhamel operator starting at time $t_0$, and
	we decompose the solution $v_N(t)$ of \eqref{cubicN} as
	$ v_N(t)=\<1>_N(t)+w_N(t),
	$
	then $w_N(t)$ solves the integral equation
	$$ w_N(t)=\Pi_N\mathcal{J}_0(\kappa(\<1>_N+w_N)+4\ov{a}_2(\<1>_N+w_N)^{\diamond 3} ).
	$$
	The remainder $w(t)$ is pretended to be in a more regular space $L_t^{\infty}H_x^s$ with $s=s_0-\epsilon$.
	For $q\in[2,\infty)$, by the large deviation estimate, $R$-certainly, i.e. outside a set of $\mu$-measure $<e^{-cR^{c'}}$, we have
			$$ \|\<1>_N^{\diamond l}\|_{L_t^qW_x^{-(3-l)\beta-\eps,\infty}([0,1])}\leq R,\quad l=1,2,3.
			$$
	by Lemma \ref{le:wave_cubic}, for $\tau\ll R^{-\frac{4}{3}}\ll 1$, 
	$$ \big\|\Pi_N\mathcal{J}_0(\kappa(\<1>_N+w_N)+4\ov{a}_2(\<1>_N+w_N)^{\diamond 3} )\big\|_{L_t^{\infty}H_x^s([0,\tau]\times\T^2)}\leq CR\tau^{\frac{3}{4}}\ll 1.
	$$
	Therefore, $R$-certainly we have local well-posedness on $[0,\tau]$, with a reminder $w_N\in Y^s([0,\tau])$ as well as the  convergence $w_N\rightarrow w$ in $Y^s([0,\tau])$ for $s=4\alpha-3-\eps$. To iterate the local well-posedness (convergence) to a long time interval, we make use of the invariance of the Gibbs measure 
	$$ \widetilde{\nu}_N(d\phi):=\exp\Big(-\int_{\T^2}\kappa (\Pi_N\phi)^{\diamond 2}+4\ov{a}_2(\Pi_N\phi)^{\diamond 4}\Big)\mu_{\alpha}(d\phi).
	$$ 
	Though the sign of $\kappa$ may not be positive, due to the defocusing nature $\ov{a}_2>0$, $\widetilde{\nu}_N\rightarrow \nu$, the Gibbs measure associated to \eqref{cubic}. The rest globalization argument is standard (see for example \cite{STz}) and we omit the detail. Furthermore, we have the invariance of $\vec{\nu}:=\nu\otimes\mu'$ along the flow $\vec{\Phi}(t)$ of \eqref{cubic}. To summarize, the version of well-posedness for the cubic equation is as follows:

	\begin{prop} \label{th:wave_cubicdetailed}
		Let $T>0$, $\alpha \in \big(\frac{8}{9},1\big),0<\eps\ll 1,$ be given. Assume that  and $s=s_0-\eps$. Then there exists a measurable set $\Sigma_0 \subset \mathcal{H}^{-\beta-\eps}$ with $\vec{\mu}(\Sigma_0) = 1$ and a flow map
		\begin{equation*}
			\vec{\Phi}(t) = \big( \Phi(t), \Phi'(t) \big)
		\end{equation*}
		defined on $\Sigma_0$ with the following properties:
		\begin{enumerate}
			
			\item $u(t):=\Phi(t)\vec{\phi}$ is the unique limit in $C([0,T];H^{-\beta-\eps}(\T^2))$ of the sequence of smooth solutions $v_N$ of
		\eqref{cubicN}.
			\item $\vec{\Phi}(t)(\Sigma_0) = \Sigma_0$ for every $t \in \R$ and the flow property holds for $\vec{\Phi}(t)$. 
			\item The measure $\vec{\nu}$ is invariant under the flow $\vec{\Phi}(t)$.
			\item For every $\vec{\phi} \in \Sigma_0$ the function
			\begin{equation*}
				(w(t),\partial_tw(t)) := \vec{\Phi}(t) \vec{\phi} - \vec{\<1>}(\phi)
			\end{equation*}
			solves the equation
			\begin{equation*}
				\begin{cases}
					&\d_t^2 w + (\dD^{\alpha})^2w+\kappa w+ 4\ov{a}_2  \big( \<1>^{\diamond 3} + 3 \<1>^{\diamond 2} w + 3 \<1> w^2 + w^3 \big) = 0 \;,\\
					&\big( w(0), \d_t w(0) \big) = (0,0) \;.
				\end{cases}
			\end{equation*}
			in $\cC([0,T]; H^{s}(\T^2))\cap L_t^{2+\theta_0}L_x^{\frac{2(\theta_0+2)}{\theta_0}}([0,T]\times\T^2)$ in the sense that the corresponding Duhamel formula holds. Furthermore, the random object $\<1>(\phi)$ verifies
			$$ \|\<1>(\phi)^{\diamond l}\|_{L_t^{10m}W_x^{-l\beta-\eps,\infty}([0,T]) }<\infty,\quad l=1,2,3. 
			$$
		\end{enumerate}
	\end{prop}


	\subsection{Convergence of higher order systems}
	Now we study the dynamical weak universality problem by proving the following result which leads to Theorem \ref{thm:dynamics}:
		\begin{prop}\label{convergenceN}
		Let $T>0,\alpha\in \big(\frac{8}{9},1\big), 0<\eps\ll 1$. Let $s=s_0-\eps$ and $s_1=s-2\eps$. Then there exists a full $\vec{\mu}$ measure set $\Sigma\subset\mathcal{H}^{-\beta-\eps}$, such that for any $\vec{\phi}\in\Sigma$, the solutions $\vec{u}_N(t)=\vec{\Phi}_N(t)\vec{\phi}$ of \eqref{Universality-N} admit a decomposition $\vec{u}_N(t)=\vec{\<1>}_N(t)+\vec{w}_N(t)$ and converge in $C([0,T];\mathcal{H}^{-\beta-\eps})$ to the solution  $\vec{\Phi}(t)\vec{\phi}$ of the cubic equation constructed in Proposition \ref{th:wave_cubicdetailed}. Moreover, the nonlinear remainders $w_N(t)$ converge in a smoother space:
		$$ \lim_{N\rightarrow\infty}\|w_N(t)-w(t)\|_{L_{t}^{\infty}H_x^{s_1}([0,T]) }=0.
		$$
	\end{prop}

	 The main ingredient to prove the almost sure convergence of \eqref{Universality-N'} to \eqref{cubicN} in $C([0,T];H^{-\beta-\eps}(\T^2))$ is a variant of the Bourgain-Bulut type argument (\cite{BB}). Briefly, we will use two global information, the first one is the invariance of measures $\vec{\nu}_N$ along the truncated flow $\vec{\Phi}_N(t)$. This will allow us to essentially control the $L_x^{\infty}$ norm of the solution $\vec{\Phi}_N(t)\vec{\phi}$ by $N^{\beta+}$. The second one is the solution of the cubic equation, thanks to Proposition \ref{th:wave_cubicdetailed}. Technically, since we deal with solutions in the space of negative regularity, it would be more convenient to work with the nonlinear part of the flow that leaves in the spaces of positive regularity.
	
	Writing
	$$ u_N=\<1>_N+w_N,
	$$
	we expand the nonlinearity $\kappa_N\Pi_Nu_N+ \Pi_N W_N'(\Pi_N u_N)$ as
	\begin{align}\label{decomposition}  \kappa_N(\<1>_N+w_N)+4\sum_{l=0}^3 \binom{3}{l}\ov{a}_{2,N}w_N^l\<1>_N^{\diamond 3-l}+\sum_{l=0}^{2m-1}\mathcal{R}_{N,l}w_N^l,
	\end{align}
	where
	\begin{align}\label{RlN} 
		\mathcal{R}_{N,l}=\sum_{j = 2 \vee (\lfl \frac{\ell}{2} \rfl + 1)}^{m}\frac{(2j)!}{l!(2j-l-1)!}\ov{a}_{j,N}N^{-(2j-4)\beta}\<1>_N^{\diamond (2j-l-1)}.
	\end{align}

	\subsubsection{Large deviation estimates}
	First, we prove the following lemma that allows us to pass from $\nu_N$ measure to $\mu$:
	\begin{lem}\label{large-deviationVN}
		For any $R>0$ and $N\in\N$,
		$$ \mu\Big\{\phi:\Big|\int_{\T^2}V_N(\Pi_N\phi)dx\Big|>R \Big\}\leq e^{-cR^{\frac{1}{2m}}}.
		$$
	\end{lem}
	\begin{proof}
		Since $\int_{\T^2}V_N(\Pi_N\phi)dx$ is a linear combination of multi-linear Gaussians of degree smaller than or equal to $2m$, by the Wiener-chaos estimate
		\begin{align}\label{EmuLp} 
			\Big(\mathbb{E}^{\mu}\Big[\Big|\int_{\T^2}V_N(\Pi_N\phi)dx\Big|^p\Big]\Big)^{\frac{1}{p}}\leq Cp^m\Big(\mathbb{E}^{\mu}\Big[\Big|\int_{\T^2}V_N(\Pi_N\phi)dx\Big|^2\Big]\Big)^{\frac{1}{2}}
		\end{align}
		for any $p\geq 2$. Using the identity (see \eqref{esperance-scalaire}) 
		$$ \mathbb{E}^{\mu}[(\Pi_N\phi)^{\diamond k}(x)\cdot (\Pi_N\phi)^{\diamond j}(y)]=k!\delta_{kj}\Big(\sum_{|k|\leq N}\frac{1}{\langle k\rangle^{2\alpha}}\mathrm{e}^{ik\cdot(x-y)}\Big)^4,
		$$
		we deduce that
		\begin{align*}
			\mathbb{E}^{\mu}\Big[\Big|\int_{\T^2}V_N(\Pi_N\phi)dx \Big|^2 \Big]=&\sum_{l=1}^m|\ov{a}_{l,N}|^2N^{-4(l-2)\beta}l!\cdot\sum_{\substack{ 
					k_1+k_2+k_3+k_4=0\\
					|k_j|\leq N
				}
			}\prod_{j=1}^4\frac{1}{\langle k_j\rangle^{2\alpha}}.
		\end{align*}
		By Lemma~\ref{lem:convolutionBasic} and the fact that $\alpha>\frac{3}{4}$, the quantity
		$$ \sum_{\substack{ 
				k_1+k_2+k_3+k_4=0\\
				|k_j|\leq N
			}
		}\prod_{j=1}^4\frac{1}{\langle k_j\rangle^{2\alpha}}
		$$
		is uniformly bounded in $N$. This implies that the right hand side of \eqref{EmuLp} is bounded by $Cp^m$. The desired estimate then follows from the Chebyshev's inequality. 
	\end{proof}
	The following Lemma crucially uses the invariance of the measure $\vec{\nu_N}$, in the spirit of Bourgain-Bulut:
	\begin{lem}\label{newlargedeviation}
		Let $T>0,\gamma>\beta=1-\alpha$ and $2\leq q,r<\infty$. There exist two positive constants $C_{T,\gamma,q,r},c_{T,\gamma,q,r}$ such that for all $\lambda>1$, $M<N$,
		\begin{align*}
			\vec{\mu}\Big(\{\vec{\phi}: \|\pi_M^{\perp}\vec{\Phi}_N(t)\vec{\phi}\|_{L_t^q\mathcal{W}_x^{-\gamma,r}([0,T]) }>\lambda  \}\Big)\leq C_{T,\gamma,q,r}\exp\big(-(T^{-\frac{1}{q}}M^{\gamma-\beta}\lambda)^{c_{T,\gamma,q,r}}\big).
		\end{align*}
		Here $\pi^{\perp}_M=\mathrm{Id}-\pi_M$ and $\pi_M$ is some smooth cutoff\footnote{The same statement holds if we replace the smooth cutoff $\pi_M$ by $\Pi_M$. Here we state the lemma with $\pi_M$ since $\Pi_M$ is not bounded in $L^p(\T^2)$, $1<p<\infty$. }.  
	\end{lem}		
	
	\begin{proof}
		In the proof, we denote by $\langle\vec{\nabla}\rangle^{-\gamma}:=\langle\nabla\rangle^{-\gamma}\otimes \langle\nabla\rangle^{-\gamma-\alpha}$. The notation $L_t^q\mathcal{X}_x$ will stand for $L_t^q\mathcal{X}_x([0,T])$.
		
		Take a parameter $\lambda_1>0$ to be fixed later, we have
		\begin{align*}
			\vec{\mu}\Big\{\vec{\phi}:\|\pi_M^{\perp}\vec{\Phi}_N(t)\vec{\phi} \|_{L_t^q\mathcal{W}_x^{-\gamma,r} }>\lambda \Big\} \leq 
			&\underbrace{\vec{\mu}\Big\{\vec{\phi}:\|\pi_M^{\perp}\vec{\Phi}_N(t)\vec{\phi} \|_{L_t^q\mathcal{W}_x^{-\gamma,r} }>\lambda, \int_{\T^2}V_N(\Pi_N\phi)dx\leq \lambda_1 \Big\}}_{\mathrm{I}}\\
			+&\underbrace{\vec{\mu}\Big\{\vec{\phi}:\|\pi_M^{\perp}\vec{\Phi}_N(t)\vec{\phi} \|_{L_t^q\mathcal{W}_x^{-\gamma,r} }>\lambda, \int_{\T^2}V_N(\Pi_N\phi)dx> \lambda_1 \Big\}}_{\mathrm{II}}.
		\end{align*} 
		By Lemma \ref{large-deviationVN},
		\begin{align}\label{II} 
			\mathrm{II}\leq e^{-c\lambda_1^{\frac{1}{2m}}}.
		\end{align}
		To estimate I, we recall that $$\vec{\nu}_N(d\vec{\phi})=\frac{1}{\mathcal{Z}_N}e^{-\int_{\T^2}V_N(\Pi_N\phi)}\vec{\mu}(d\vec{\phi}),$$
		then
		\begin{align*}
			\mathrm{I}\leq \mathcal{Z}_Ne^{\lambda_1}\vec{\nu}_N\Big\{\vec{\phi}:\|\pi_M^{\perp}\vec{\Phi}_N(t)\vec{\phi} \|_{L_t^q\mathcal{W}_x^{-\gamma,r} }>\lambda \Big\}.
		\end{align*}
		Take $q_1\geq \max\{q,r\}$ to be specified, by Chebyshev's inequality and Minkowski's inequality, we have
		\begin{align*}
			\mathrm{I}\leq \frac{\mathcal{Z}_Ne^{\lambda_1}}{\lambda^{q_1}}\Big\|
			\Big(\int_{\mathcal{H}^{-\beta_{\eps}}} |\langle\vec{\nabla}\rangle^{-\gamma}\pi_M^{\perp}(\vec{\Phi}_N(t)\vec{\phi})  |^{q_1}\vec{\nu}_N(d\vec{\phi}) \Big)^{\frac{1}{q_1}}
			\Big\|_{L_t^qL_x^r}^{q_1}.
		\end{align*}
		By the invariance of $\vec{\nu}_N$ along $\vec{\Phi}_N(t)$, we deduce that, for a.e.$x\in\T$ and $t\in[0,T]$ (see Lemma 7.1 of \cite{STz} for a rigorous proof)
		$$ \int_{\mathcal{H}^{-\beta_{\eps}}} |\langle\vec{\nabla}\rangle^{-\gamma}\pi_M^{\perp}(\vec{\Phi}_N(t)\vec{\phi})  |^{q_1}(x)\vec{\nu}_N(d\vec{\phi})=\int_{\mathcal{H}^{-\beta_{\eps}}} |\langle\vec{\nabla}\rangle^{-\gamma}\pi_M^{\perp}(\vec{\phi})  |^{q_1}(x)\vec{\nu}_N(d\vec{\phi}).
		$$
		Hence 
		\begin{align*}
			\mathrm{I}\leq \frac{\mathcal{Z}_Ne^{\lambda_1}T^{\frac{q_1}{q}}}{\lambda^{q_1}}\Big\|
			\int_{\mathcal{H}^{-\beta_{\eps}}} |\langle\vec{\nabla}\rangle^{-\gamma}\pi_M^{\perp}\vec{\phi}  |^{q_1}\vec{\nu}_N(d\vec{\phi}) \Big\|_{L_x^{\frac{r}{q_1}}}.
		\end{align*}
		By Cauchy-Schwarz, the boundedness of $\mathcal{Z}_N,\mathcal{Z}_N^{-1}$ and Proposition \ref{pr:uniform_bd_density}, the above quantity can be controlled by
		$$ \frac{Ce^{\lambda_1}T^{\frac{q_1}{q}}}{\lambda^{q_1}}\Big\|\Big(\int_{\mathcal{H}^{-\beta_{\eps}}} |\langle\vec{\nabla}\rangle^{-\gamma}\pi_M^{\perp}\vec{\phi} |^{2q_1}\vec{\mu}(d\vec{\phi}) \Big)^{\frac{1}{2}}\Big\|_{L_x^{\frac{r}{q_1}}}\leq C^{q_1}T^{\frac{q_1}{q}}e^{\lambda_1}\frac{q_1^{\frac{q_1}{2}}M^{-(\gamma-\beta)q_1} }{\lambda^{q_1}}.
		$$
		So for any $q_1\geq q,r, \lambda_1\leq \lambda$, we have
		\begin{align*}
			\mathrm{I}+\mathrm{II}\leq e^{\lambda_1}\Big(\frac{CT^{\frac{1}{q}}\sqrt{q}M^{-(\gamma-\beta)} }{\lambda}\Big)^{q_1}+e^{-\lambda^{\frac{1}{2m}}}.
		\end{align*}
		By optimizing the choice of $\lambda_1,q_1$, we complete the proof of Lemma \ref{newlargedeviation}.
	\end{proof}	
The following Lemma consists of key arguments of the proof of Proposition \ref{convergenceN}.
	\begin{lem}\label{Firstcomparison}
		Let $T\geq 1$, $\eps\ll 1$. Let $R\gg 1, N\gg 1$ be large parameters.
		Assume that $\vec{\phi}\in\mathcal{H}^{-\beta-\eps}$ satisfies
		\begin{align}\label{globalbound} 
			\|\vec{\Phi}_N(t)\vec{\phi}\|_{L_t^{10m}\mathcal{W}_x^{-\beta-\eps,\infty}}\leq R,\quad \|\vec{\<1>}_N\|_{L_t^{10m}\mathcal{W}_x^{-\beta-\eps,\infty} }\leq R,
		\end{align}
		and 
		\begin{align*}
			&\|\<1>_N^{\diamond k}\|_{L_{t}^{10m}L_x^{\infty}}\leq N^{k\beta+\eps},\quad\quad\quad\quad\quad\quad\quad\quad\;\; \|\<1>_N^{\diamond l}\|_{L_{t}^{10m}W_x^{-l\beta-2\eps,\frac{1}{\eps}}  }\leq R,\\
			& 
		 \|\<1>_N^{\diamond (n-l-1)}\|_{L_t^{10m}W_x^{-(3-l)\beta-2\eps,\frac{1}{\eps}} }\leq N^{(n-4)\beta-\eps}, \quad
			\|\<1>_{N}^{\diamond l}-\<1>^{\diamond l} \|_{L_{t}^{10m}\mathcal{W}_x^{-l\beta-2\eps,\frac{1}{\eps}} }\leq N^{-\frac{\eps}{2}},
		\end{align*}
		for all $1\leq k\leq 2m-1$, $4\leq n\leq 2m-1$ and $l\in\{1,2,3\}$, where $L_t^{q}\mathcal{X}$ stands for $L^{q}([0,T];\mathcal{X})$.
		Moreover, assume that on $[0,T]$, for all $1\leq l\leq 3$,
		\begin{align}\label{cubicbd}  \sum_{l=1}^3\|\<1>^{\diamond l}\|_{L_t^{10m}W_x^{-l\beta-\eps,\infty}([0,T]) }+ \|\vec{\Phi}(t)\vec{\phi}-\vec{\<1>}\|_{\mathcal{Y}^s([0,T])} \leq R.
		\end{align}
		Then for any $\frac{2-\alpha}{2+\theta_0}<s_1=s_0-2\eps$\footnote{Under the constraint $\alpha\in(\frac{8}{9},1)$, for $0<\theta_0\ll \epsilon\ll 1$, $s_1>\frac{2-\alpha}{2+\theta_0}$. }, there exist constants $C=C_{m,\eps,\beta,s_1}>0$ and $K_0>0$, such that if the parameters $R, N$ satisfy the constraint
		$$ (K_0)^{TR^{100m}}<N^{\frac{\eps}{2}}, \text{ or equivalently, } R<\Big(\frac{\eps \log N}{2T\log K_0}\Big)^{\frac{1}{100m}},
		$$
		then
		$$ \|(\vec{\Phi}_N(t)\vec{\phi}-\vec{\<1>}_N)-(\vec{\Phi}(t)\vec{\phi}-\vec{\<1>})\|_{\mathcal{Y}^{s_1}([0,T])}
		\leq C_{\eps}N^{-\frac{\eps}{4}}.
		$$
	\end{lem}
	\begin{proof}

		We write
		$$ u_N(t)=\Phi_N(t)\vec{\phi}=\<1>_N+w_N(t),\quad u(t)=\Phi(t)\vec{\phi}=\<1>+w(t).
		$$
		By \eqref{globalbound}, \eqref{cubicbd} and Bernstein, we deduce that
		\begin{align}\label{basebound}  \|w_N(t)\|_{L_t^{10m}L_x^{\infty}([0,T])}\leq CN^{\beta+2\eps}R,\quad  \|\vec{w}(t)\|_{\mathcal{Y}^s([0,T]) }\leq R.
		\end{align}
		
		\noi
		$\bullet${\bf Step 1: Recursive inequality}

		Fix $t_0\in[0,T-\tau_0]$ and $I_{t_0,\tau_0}:=[t_0,t_0+\tau_0]$, where $\tau_0$ is a small parameter to be chosen later.
			Throughout the proof, the symbol $A \lesssim B$ stands for $A\leq CB$ for some constant $C$ that is \emph{independent of} parameters $R,N,\tau_0,t_0$.

		By the Strichartz inequality \eqref{Strichartz2d}, we have
		\begin{align}\label{recurrence'} 
			\|\vec{w}_N(t)-\vec{w}(t)\|_{\mathcal{Y}^{s_1}(I_{t_0,\tau_0}) } \lesssim& \|w_N(t_0)-w(t_0)\|_{H_x^{s_1}}\\+ &\Big[A_{N}(I_{t_0,\tau_0})+\sum_{l=4}^{2m-1}B_{N,l}(I_{t_0,\tau_0})+\sum_{l=0}^3\sum_{j=3}^mC_{N,j,l}(I_{t_0,\tau_0})\Big],
		\end{align}
		where
		\begin{align*}
			&A_N(I_{t_0,\tau_0}):=|\lambda-\ov{a}_{2,N}|\|\<1>^{\diamond 3} + 3 \<1>^{\diamond 2} w + 3 \<1> w^2 + w^3  \|_{L_{t}^1H_x^{s_1-\alpha}(I_{t_0,\tau_0}) }+\|\<1>_N^{\diamond 3}-\<1>^{\diamond 3} \|_{L_{t}^1H_x^{s_1-\alpha}(I_{t_0,\tau_0}) } \\
			&\hspace{2cm}+ \|\<1>_N^{\diamond 2}\cdot w_N-\<1>^{\diamond 2}w \big\|_{L_{t}^{1}H_x^{s_1-\alpha}(I_{t_0,\tau_0}) }
			+\|\<1>_Nw_N^2-\<1>\cdot w^2 \|_{L_{t}^{1}H_x^{s_1-\alpha}(I_{t_0,\tau_0}) }\\&\hspace{2cm}+\|w_N^3-w^3\|_{L_{I_{t}}^{1}H_x^{s_1-\alpha}(I_{t_0,\tau_0}) }+\|\Pi_N^{\perp}\vec{w}\|_{L_{t}^{\infty}\mathcal{H}_x^{s_1}(I_{t_0,\tau_0}) }+\tau_R|\kappa_N-\kappa|\|\vec{w}\|_{L_t^{1}\mathcal{H}_x^{s_1}(I_{t_0,\tau_0})}\\&\hspace{2cm}+\tau_R|\kappa_N|\|\vec{w}_N-\vec{w}\|_{L_t^{\infty}\mathcal{H}_x^{s_1}(I_{t_0,\tau_0}) } ,\\
			&B_{N,l}:=\|\mathcal{R}_{l,N}\cdot w_N^l \|_{L_{t}^{1}H_x^{s_1-\alpha}(I_{t_0,\tau_0}) },\quad 4\leq l\leq 2m-1  \\
			&C_{N,j,l}:=N^{-(2j-4)\beta}\|\<1>_N^{\diamond(2j-1-l)}\cdot w_N^l \|_{L_{t}^{1}H_x^{s_1-\alpha}(I_{t_0,\tau_0}) },\quad 0\leq l\leq 3\text{ and } 3\leq j\leq m.	
		\end{align*}

		From Lemma \ref{le:wave_cubic}, we have for sufficiently small $\eps>0$ and $q>1$ large enough,
		\begin{align}\label{b1}
			\|\<1>_N^{\diamond 2}w_N-\<1>^{\diamond 2}w\|_{L_{t}^1H_x^{s_1-\alpha}(I_{t_0,\tau_0}) }\lesssim_{\eps} & \|\<1>_N^{\diamond 2}-\<1>^{\diamond 2}\|_{L_{t}^1W_x^{-2\beta-2\eps,\frac{1}{\eps}}(I_{t_0,\tau_0}) }\|w\|_{L_{t}^{\infty}H_x^{s_1}(I_{t_0,\tau_0})}\notag \\
			+&\tau_0^{\frac{1}{2}}\|\<1>_N^{\diamond 2}\|_{L_{t}^{2}W_x^{-2\beta-2\eps,\frac{1}{\eps}}(I_{t_0,\tau_0}) }\|w_N-w\|_{L_{t}^{\infty}H_x^{s_1}(I_{t_0,\tau_0}) }.
		\end{align}
		\begin{align}\label{b2} 
			\|\<1>_Nw_N^2-\<1>w^2\|_{L_{t}^{1}H_x^{s_1-\alpha}(I_{t_0,\tau_0}) }\lesssim_{\eps} &\|\<1>_N-\<1>\|_{L_{1}^1W_x^{-\beta-2\eps,\frac{1}{\eps}}(I_{t_0,\tau_0}) }\|w\|_{L_{t}^{\infty}H_x^{s_1}(I_{t_0,\tau_0}) }^2\notag \\
			+&\tau_0^{\frac{1}{2}}\|\<1>_N\|_{L_{t}^{4}W_x^{-\beta-2\eps,\frac{1}{\eps}}(I_{t_0,\tau_0}) }\|w_N+w\|_{L_{t}^{4}H_x^{s_1}(I_{t_0,\tau_0}) }
			\|w_N-w\|_{L_{t}^{\infty}H_x^{s_1}(I_{t_0,\tau_0}) }
		\end{align}
		and
		\begin{align}\label{b3} 
			\|w_N^3-w^3\|_{L_t^1H_x^{s_1-\alpha}(I_{t_0,\tau_0})}\lesssim &\tau_0^{\frac{1}{2}}
			\|w_N-w\|_{L_{t}^{\infty}H_x^{s_1}(I_{t_0,\tau_0})}\big(\|w_N\|_{L_{t}^{4}H_x^{s_1}(I_{t_0,\tau_0})}^2+
			\|w\|_{L^{4}_{t}H_x^{s_1}(I_{t_0,\tau_0}) }^2 \big),
		\end{align}
		where all the implicit constants are independent of $N, R,\tau_0,t_0$, but can depend on $m$ and $\eps$. Therefore,
		\begin{align}\label{ANbound} A_{N}(I_{t_0,\tau_0})\lesssim D_N(I_{t_0,\tau_0})+\tau_0^{\frac{1}{2}}F_N(I_{t_0,\tau_0})\|w_N-w\|_{L_t^{\infty}H_x^{s_1}(I_{t_0,\tau_0}) }+\tau_0|\kappa_N-\kappa|\|\vec{w}\|_{L_t^{\infty}\mathcal{H}_x^{s_1}(I_{t_0,\tau_0})},
		\end{align}
		where
		\begin{align*}
			D_N(I_{t_0,\tau_0}):=&|\lambda-\ov{a}_{2,N}|\|\<1>^{\diamond 3} + 3 \<1>^{\diamond 2} w + 3 \<1> w^2 + w^3  \|_{L_{t}^1H_x^{s_1-\alpha}(I_{t_0,\tau_0}) }+\|\<1>_N^{\diamond 3}-\<1>^{\diamond 3} \|_{L_{t}^1H_x^{s_1-\alpha} (I_{t_0,\tau_0})}\\
			+&\|\<1>_N^{\diamond 2}-\<1>^{\diamond 2}\|_{L_{t}^1W_x^{-2\beta-\eps,\frac{1}{\eps}}(I_{t_0,\tau_0}) }\|w\|_{L_{t}^{\infty}H_x^{s_1}(I_{t_0,\tau_0})
			}\notag 
			+\|\<1>_N-\<1>\|_{L_{t}^1W_x^{-\beta-\eps,\frac{1}{\eps}}(I_{t_0,\tau_0}) }\|w\|_{L_{t}^{\infty}H_x^{s_1}(I_{t_0,\tau_0})
			}^2\notag\\
			+&\|\Pi_N^{\perp}w\|_{L_{t}^{\infty}H_x^{s_1}(I_{t_0,\tau_0})
			},\\
			F_N(I_{t_0,\tau_0})=:&\|\<1>_N^{\diamond 2}\|_{L_{t}^1W_x^{-2\beta-\eps,\frac{1}{\eps}}(I_{t_0,\tau_0}) }+
			\|\<1>_N\|_{L_{t}^1W_x^{-\beta-\eps,\frac{1}{\eps}}(I_{t_0,\tau_0})
			}
			\big(\|w_N\|_{L_{t}^{\infty}H_x^{s_1}(I_{t_0,\tau_0}) }+
			\|w\|_{L_{t}^{\infty}H_x^{s_1}(I_{t_0,\tau_0}) }
			\big)\\
			+&\|w_N\|_{L_{t}^{\infty}H_x^{s_1}(I_{t_0,\tau_0}) }^2+\|w\|_{L_{t}^{\infty}H_x^{s_1}(I_{t_0,\tau_0}) }^2+1.
		\end{align*}
		Applying Lemma \ref{le:wave_cubic}, Cauchy-Schwarz and the fact that
		$ \frac{1}{\eps}\leq 2+\frac{4}{\theta_0},
		$		 we have (here it is important that $l\geq 4$)
		\begin{align}\label{BNbound}
			&B_{N,l}(I_{t_0,\tau_0})\\ \lesssim_{\eps} & \sum_{\frac{l+1}{2}\leq j\leq m}N^{-(2j-4)\beta}\|\<1>_N^{\diamond (2j-1-l)}w_N^{l-3} \|_{L_t^1L_x^{2+\frac{4}{\theta_0}}(I_{t_0,\tau_0})}\|w_N\|_{L_t^{\infty}H_x^{s_1}(I_{t_0,\tau_0}) }^3\notag \\
			\lesssim_{\eps,m} &\tau_0^{\frac{1}{4}}N^{-\beta}\|w_N\|_{L_t^{\infty}H_x^{s_1}(I_{t_0,\tau_0}) }^3\|w_N\|_{L_t^2L_x^{2+\frac{4}{\theta_0}}(I_{t_0,\tau_0}) }\sup_{\frac{l+1}{2}\leq j\leq m}N^{-(2j-5)\beta}\|\<1>_N^{\diamond (2j-l-1)}w_N^{l-4}\|_{L_t^{4}L_x^{\infty}(I_{t_0,\tau_0})}\notag \\
			\lesssim &\tau_0^{\frac{1}{4}}N^{-\beta}
			\|w_N\|_{Y^{s_1}(I_{t_0,\tau_0}) }^4
			\sup_{2\leq j\leq m}N^{-(2j-5)\beta}\|\<1>_N^{\diamond (2j-l-1)}\|_{L_t^{10m}L_x^{\infty}(I_{t_0,\tau_0})}\|w_N\|_{L_t^{10m}L_x^{\infty}}^{l-4}\notag \\
			\lesssim &
			 \tau_0^{\frac{1}{4}}N^{-\beta+2m\eps}
			 \|w_N\|_{Y^{s_1}(I_{t_0,\tau_0}) }^4,
			 \notag 
		\end{align}
	where to the last step, we have used \eqref{basebound} and the $L_t^{10m}L_x^{\infty}$ bound for $\<1>_N^{\diamond (k)}$. 
		Note that here it is crucial to put one $w_N$ in the space $L_t^{2+\theta_0}L_x^{2+\frac{4}{\theta_0}}$ in order to gain some negative power of $N$, as putting \eqref{basebound} on all $w_N$ will lead to a bound $N^{l\eps}$ that does not converge to zero as $N\rightarrow\infty$. 
		
		Similarly,
		\begin{align}\label{CNbound}
			C_{N,j,l}(I_{t_0,\tau_0})\lesssim_{\eps,m}N^{-(2j-4)\beta}\tau_0^{\frac{1}{2}}\|\<1>_N^{\diamond (2j-l-1)}\|_{L_t^{2}W_x^{-(3-l)\beta-2\eps,\frac{1}{\eps}}(I_{t_0,\tau_0}) } \|w_N\|_{L_{t}^{\infty}H_x^{s_1}(I_{t_0,\tau_0}) }^l.
		\end{align}
		
		\noi
		$\bullet${\bf Step 2: Bootstrap argument}
		
		We first claim that if for some $T_1\in(0,T)$,
		\begin{align}\label{bootstrapbound1}  \|w_N\|_{Y^{s_1}([0,T_1]) }\leq 2R^{10},
		\end{align}
		then for $R,N$ large enough, there exist $C_{\eps}>0$ and absolute constant $K_0>0$, such that 
		\begin{align}\label{claim}  \|\vec{w}_N-\vec{w}\|_{\mathcal{Y}^{s_1}([0,T_1]) }\leq C_{\eps}K_0^{TR^{50m}}N^{-\frac{\eps}{2}}.
		\end{align}		
		Indeed, we decompose $[0,T_1]$ into $k_0$ intervals of size $\tau_0=\tau_0(R)=R^{-100m}$, and denote by 
		$$\mathbf{x}_k:=\|\vec{w}_N-\vec{w}\|_{\mathcal{Y}^{s_1}(J_k) },
		$$
		where $J_k=(k\tau_0,(k+1)\tau_0]$.
		By \eqref{recurrence'}, \eqref{ANbound},\eqref{BNbound},\eqref{CNbound}, \eqref{basebound}, we deduce that
		\begin{align*}
			\mathbf{x}_k\leq C_{\eps}\tau_0^{\frac{1}{4}}R^{20m}\mathbf{x}_k+C_0\mathbf{x}_{k-1}+C_{\eps}R^{3}\tau_0N^{-\frac{\eps}{2}}+C_{\eps}\tau_0^{\frac{1}{2}}R^{50m}N^{-\beta+2m\eps}.
		\end{align*}
		For $R$ large enough, $\tau_0$ small enough such that $$C_{\eps}\tau_0^{\frac{1}{4}}R^{20m}<\frac{1}{2},$$ we deduce that (provided that $\eps<\beta/2m$)
		$$ \mathbf{x}_k\leq 2C_0\mathbf{x}_{k-1}+C_{\eps}\tau_0^{\frac{1}{2}}R^{50m}N^{-\frac{\eps}{2}}. 
		$$
		This yields
		$$ \mathbf{x}_k\leq (2C_0)^{\frac{T_1}{\tau_0}}\mathbf{x}_0+(2C_0)^{\frac{T_1}{\tau_0}}C_{\eps}T_1\tau_0^{\frac{1}{2}}R^{50m}N^{-\frac{\eps}{2}}\leq C_{\eps}(2C_0)^{\frac{T}{\tau_0}}T\tau_0^{\frac{1}{2}}R^{50m}N^{-\frac{\eps}{2}}.
		$$
		Hence \eqref{claim} follows.
		
		To finish the proof, it suffices to prove the bootstrap assumption \eqref{bootstrapbound1} up to time $T_1=T$, with slightly smaller upper bound $R^{10}$ instead of $2R^{10}$. More precisely, Let $T_*\leq T_1$ be the largest number such that 
		$$\|w_N\|_{Y^{s_1}([0,T_*]) }\leq R^{10}.
		$$
		Since for fixed $N$, $w_N$ solves an ODE in the finite-dimensional space, we deduce that the function
		$$ t\mapsto \|w_N\|_{Y^{s_1}([0,t]) }
		$$ 
		is continuous, thus $T_*>0$. On the other hand, if $T_*<T_1$, again by continuity, there exists $\delta_*\in(0,T_1-T_*)$, such that
		$$  \|w_N\|_{Y^{s_1}([0,T_*+\delta_*]) }<R^{10}+1<2R^{10}.
		$$
		Therefore, we deduce that \eqref{claim} holds with $T_1=[0,T_*+\delta_*]$. In particular,
		\begin{align*}  \|w_N\|_{Y^{s_1}([0,T_*+\delta_*]) }\leq R+C_{\eps}(K_0)^{T_0R^{100m}}N^{-\frac{\eps}{2}}<R(1+C_{\eps}N^{-\frac{\eps}{4}})<2R,
		\end{align*}
		provided that $N$ is large enough such that $C_{\eps}N^{-\frac{\eps}{4}}<1$. This contradicts to the definition of $T_*$.  So we must have $T_*=T$. The proof of Lemma \ref{Firstcomparison} is complete.		
	\end{proof}
\vspace{0.3cm}

\noi
$\bullet$
{\bf Proof of Proposition \ref{convergenceN}}:
	
First we note that by choosing $R_N=(\log N)^{\theta}$ for $\theta\ll 1$, Lemma \ref{Firstcomparison} allows to prove the almost sure convergence of the dyadic sequence. To prove the convergence of the full sequence, we first define properly the good data set. For each dyadic number $N$, let $R_N=(\log N)^{\theta}$, $M_N=(\log N)^{A_0}$ for $A_0\gg 1,0<\theta\ll 1$. Define
	\begin{align*} 
	 \Sigma_{1,N}:=&\bigcap_{l=0}^3\bigcap_{k=4}^{2m-1}\bigcap_{N_1=N}^{2N_1}\big\{
	 \|\<1>_N^{\diamond l}-\<1>_{N_1}^{\diamond l}\|_{L_t^{10m}W_x^{-l\beta-\eps,\frac{1}{\eps}}}\leq N^{-\frac{\eps}{2}}
	 \big\}\\
	 &\hspace{1.5cm}\cap\big\{
	\|\<1>_{N_1}^{\diamond k}\|_{L_t^{10m}L_x^{\infty}}\leq N^{k\beta+\eps},\; \|\<1>_N^{\diamond (k-l)}\|_{L_t^{10m}W_x^{-(3-l)\beta-2\eps,\frac{1}{\eps}} }\leq N^{(k-3)\beta-\eps}
	 \big\},\\
	 \Sigma_{2,N}:=&\big\{ 	\|\vec{\Phi}_N(t)\vec{\phi}\|_{L_t^{10m}\mathcal{W}_x^{-\beta-\eps,\infty}}\leq R_N,\; \|\vec{\<1>}_N\|_{L_t^{10m}\mathcal{W}_x^{-\beta-\eps,\infty} }\leq R_N \big\},\\
	 \Sigma_{3,N}:=&\bigcap_{l=0}^3\big\{
	 \|\<1>^{\diamond l}\|_{L_t^{10m}W_x^{-l\beta-\eps,\frac{1}{\eps}}}+ \|\vec{\Phi}(t)\vec{\phi}-\vec{\<1>}\|_{\mathcal{Y}^s([0,T])} \leq R_N
	  \big\},\\
	 \Sigma_{4,N}:=&\bigcap_{l=0}^3\bigcap_{k=4}^{2m-1}\big\{
	 \|\<1>_N^{\diamond l}\|_{L_{t}^{10m}W_x^{-l\beta-\eps,\frac{1}{\eps} } }\leq R_N,\;
	 \|\<1>_N^{\diamond (k-l)}\|_{L_t^{10m}W_x^{-(3-l)\beta-2\eps,\frac{1}{\eps}}}\leq N^{(k-3)\beta-\eps}\big\}\\ 
&\hspace{1.5cm}\cap\big\{	 \|\<1>_{N}^{\diamond l}-\<1>^{\diamond l} \|_{L_{t}^{10m}W_x^{-l\beta-\eps,\frac{1}{\eps} }}\leq N^{-\frac{\eps}{2}}
	 \big\},\\
	 \Sigma_{5,N}:=&\{\sup_{N\leq N_1\leq 2N_1}\|\pi_{M_N}^{\perp}\vec{\Phi}_N(t)\vec{\phi}\|_{L_t^{10m}\mathcal{W}_x^{-\beta-\eps,\infty} }\leq 1 \}.
	\end{align*}
 
\begin{lem}\label{largedeviation} 
There exist $C>0$ and $\delta(\eps)>0$ such that for $i\in\{1,2,3,4,5\}$, there holds
\begin{align*}
\vec{\mu}(\Sigma_{i,N}^c)\leq CN^{-\delta(\eps)}.
\end{align*}
\end{lem} 
\begin{proof}
In order not to perturb the main line of argument, we set aside the proof of this lemma in Appendix \ref{appendix:largedeviation}
\end{proof}
Next, we define
	$$ \Sigma_N=\bigcap_{j=1}^5\Sigma_{j,N}.
	$$
 Then by Lemma \ref{largedeviation}, we have
	$$ \sum_{N\in 2^{\N}}\vec{\mu}(\Sigma_N^c)<\infty.
	$$
	Therefore, by Borel-Cantelli, the set
	$$ \Sigma:=\limsup_{j\rightarrow\infty}\Sigma_{2^j} 
	$$
	has full $\vec{\mu}$ measure, i.e. $\mu(\Sigma)=1$.
	To finish the proof, we need to show that for any $\vec{\phi}\in\Sigma$, $\vec{\Phi}_N(t)\vec{\phi}$ converges to $\vec{\Phi}(t)\vec{\phi}$ in $C([0,T];\mathcal{H}_x^{-\beta-\eps})$.
	
	By definition, there exists $N_0$ such that $\vec{\phi}\in\Sigma_N$ for all dyadic number $N\geq N_0$. Pick $N_1\in[N,2N]$, not necessarily a dyadic number, our goal is to compare $\vec{w}_{N_1}$ and $\vec{w}$ in $L_t^{\infty}\mathcal{H}_x^{s_1}$. 
	We will essentially follow the argument of the proof of Lemma \ref{Firstcomparison}, with an additional care that we do not have the bound $$
	\|w_{N_1}\|_{L_t^{10m}L_x^{\infty}([0,T]) }\leq N^{\beta+2\eps}R_N$$
	in a priori.
	 Nevertheless, the choice of $\Sigma_N$ provides a control
	$$ \|\pi_{M_N}^{\perp}w_{N_1}\|_{L_t^{\infty}W_x^{-\beta-\eps,\infty}([0,T])}\leq 1.
	$$
	Thus by the Sobolev embedding and Bernstein's inequality,
	\begin{align*}
		\|w_{N_1}\|_{L_t^{\infty}W_x^{-\beta-\eps,\infty}}\leq &\|\pi_{M_N}w_{N_1}\|_{L_t^{\infty}W_x^{-\beta-\eps,\infty}}+1\\
		\leq &\|\pi_{M_N}w_{N}\|_{L_t^{\infty}W_x^{-\beta-\eps,\infty}}+M_N\|w_{N_1}-w_N\|_{L_t^{\infty}H_x^{s_1}}+1.
	\end{align*}			
	By Bernstein again,
	\begin{align}
		\|w_{N_1}\|_{L_t^{\infty}L_x^{\infty}} \lesssim_{\eps} \|\pi_{2N_1}w_{N_1}\|_{L_t^{\infty}W_x^{2\eps,\frac{1}{\eps}}}\lesssim &N_1^{\beta+3\eps}\|w_{N_1}\|_{L_t^{\infty}W_x^{-\beta-\eps,\infty}}\notag \\ \lesssim &N^{\beta+4\eps}R_N+(\log N)^{A_0}N^{\beta+\eps}\|w_{N_1}-w\|_{L_t^{\infty}H_x^{s_1}}, \label{wN1Linfty}
	\end{align}
	for $N$ large enough.
	
	Now we argue as in the Step 2 in the proof of Lemma \ref{Firstcomparison}. Assuming first that for some $T_1\in(0,T)$, \begin{align}\label{bootstrapbound2}
		\|w_{N_1}\|_{Y^{s_1}([0,T_1]) }\leq 2R_N^{10}
	\end{align}
	holds. Consequently, we have very roughly estimate
	$$ \|\vec{w}_{N_1}-\vec{w}\|_{\mathcal{Y}^{s_1}([0,T_1])}\leq 3R_N^{10}.
	$$
	Thanks to \eqref{wN1Linfty} and the choice $R_N=(\log N)^{\theta}$, we deduce that 
	$$ \|w_{N_1}\|_{L_t^{\infty}L_x^{\infty}([0,T_1])}\leq N^{\beta+10\eps}.
	$$			
	The same iterative argument yields(by choosing $\tau_0=R_N^{-50m}$) 
	\begin{align}\label{claim2}  \|\vec{w}_{N_1}-\vec{w}\|_{\mathcal{Y}^{s_1}([0,T_1]) }\leq C_{\eps}(K_0)^{TR_N^{100m}}N^{-\frac{\eps}{2}}\leq C_{\eps}N^{-\frac{\eps}{4}},
	\end{align}	
	provided that $\theta\ll 1$ such that $R(N)^{100m}=(\log N)^{100m\theta}\ll \eps(\log N)$.
	
	Following the same bootstrap argument as in Step 2, we deduce that \eqref{claim2} is indeed true up to time $T$. This completes the proof of Proposition \ref{convergenceN}.		

	\appendix



	
	\section{Nonlinear estimates and convolution inequalities} \label{appendix:bounds}

	\begin{lem} \label{le:wave_cubic}
		Let $\alpha \in (\frac{8}{9}, 1)$ and $s \in (1-\frac{\alpha}{2}, 4\alpha-3)$. Let $\eps$ be sufficiently small such that
		\begin{equation}\label{constraint} 
			1 - \frac{\alpha}{2} + 2\eps < s < 4 \alpha - 3 - 2\eps,\; 3\alpha-2>4\eps.
		\end{equation}
		Then we have the following bounds:
		\begin{equation*}
			\begin{split}
				\|F_0\|_{H^{-(\alpha-s)}} \lesssim \|F_0\|_{H^{-3(1-\alpha)-2\eps}},\quad \|F_l u^l\|_{H^{-(\alpha-s)}} \lesssim \|F_1\|_{W^{-(3-l)(1-\alpha)-2\eps, \frac{1}{\eps}}} \|u\|_{H^s}^l
			\end{split}
		\end{equation*}
		for $l=1,2,3$.
	\end{lem}
	\begin{proof}
		
		The first inequality is trivial. To prove the second, by duality, it suffices to show that, for any $G\in H^{\alpha-s}$ such that $\|G\|_{H^{\alpha-s}}\leq 1$ and $H\in L^{\frac{1}{\eps}}$, we have
		\begin{align}\label{duality} 
			\Big|\int_{\T^2}H\cdot \langle\nabla\rangle^{(3-l)(1-\alpha)+2\eps}(u^lG)dx\Big|\lesssim \|H\|_{L^{\frac{1}{\eps}}}\|u\|_{H^s}^l.
		\end{align}
		By H\"older and Lemma \ref{le:fractional_Leibniz}, the left hand side of \eqref{duality} is bounded by
		$$ \|H\|_{L^{\frac{1}{\eps}}}\big(\|\langle\nabla\rangle^{\gamma} (u^l)\|_{L^{\frac{2}{1+\alpha-s-2\eps}}}\|G\|_{L^{\frac{2}{1-\alpha+s}}}
		+\|u^l\|_{L^{\frac{2}{1+\alpha-s-2\eps-\gamma } }} \|\nabla^{\gamma}G\|_{L^{\frac{2}{1-\alpha+s+\gamma}}}
		\big),
		$$
		where $\gamma=(3-l)(1-\alpha)+2\eps$. Using the Sobolev embedding $$H^{\alpha-s}(\T^2)\hookrightarrow L^{\frac{2}{1-\alpha+s}}(\T^2),\quad H^{\alpha-s}(\T^2)\hookrightarrow W^{\gamma,\frac{2}{1-\alpha+s+\gamma}}(\T^2),$$ the two norms of $G$ are controlled by $\|G\|_{H^{\alpha-s}}\leq 1$. Thanks to the conditions $s>1-\frac{\alpha}{2}+2\eps$ and $3\alpha-2>4\eps$, we have
		$ \frac{2l}{1+\alpha-s-\gamma-2\eps}\leq \frac{2}{1-s},
		$ for $l=1,2,3$, thus by H\"older,
		$$ \|u^l\|_{L^{\frac{2}{1+\alpha-s-2\eps-\gamma}}}\lesssim \|u\|_{H^s}^l. 
		$$
		When $l=1$, $\gamma=2(1-\alpha)+2\eps$, since $3\alpha-2\geq 4\eps$, by H\"older and Sobolev's embedding we have
		$$\|\langle \nabla\rangle^{2(1-\alpha)+2\eps}u\|_{L^{\frac{2}{1+\alpha-s-2\eps}}}\lesssim 
		\|\langle \nabla\rangle^{2(1-\alpha)+2\eps}u\|_{L^{\frac{2}{1+2(1-\alpha)+2\eps-s}}}\lesssim \|u\|_{H^s}.
		$$
		For $l=2,3$, using Lemma \ref{le:fractional_Leibniz}, H\"older's inequality and the constraint \eqref{constraint}, we get
		$$
		\|\langle\nabla\rangle^{\gamma} (u^l)\|_{L^{\frac{2}{1+\alpha-s-2\eps}}}\lesssim
		\|u\|_{W^{\gamma,\frac{2}{\alpha-2\eps-(l-2)(1-s) }}}\|u^{l-1}\|_{L^{\frac{2}{(l-1)(1-s)}}}\lesssim \|u\|_{H^s}^{l}.
		$$
		This completes the proof of Lemma \ref{le:wave_cubic}.
	\end{proof}

	\begin{lem}\label{lem:convolutionBasic}
		Let $0\leq \eta_1\leq \eta_2$ and $\eta_1+\eta_2>d$. Then there exists $C_0>0$, such that for every $k_0\in\Z^d$:
		\begin{itemize}
			\item[(i)] If $\eta_2<d$, we have
			$$  \sum_{k\in\Z^d}\frac{1}{\langle k\rangle^{\eta_1}\langle k-k_0\rangle^{\eta_2}}\leq \frac{C_0}{\langle k_0\rangle^{\eta_1+\eta_2-d}}.
			$$
			\item[(ii)] If $\eta_2=d$, then 
			$$ \sum_{k\in\Z^d}\frac{1}{\langle k\rangle^{\eta_1}\langle k-k_0\rangle^{\eta_2}}\leq \frac{C_0\log(k_0)}{\langle k_0\rangle^{\eta_1}}.
			$$
			\item[(iii)] If $\eta_2>d$, then
			$$ \sum_{k\in\Z^d}\frac{1}{\langle k\rangle^{\eta_1}\langle k-k_0\rangle^{\eta_2}}\leq \frac{C_0}{\langle k_0\rangle^{\eta_1}}.
			$$
			\item[(iv)]	If $\frac{(n-1)d}{n} < \eta < d$, then we have the bound
			\begin{equation*}
				\sum_{\substack{(k_1,\cdots,k_n)\in(\Z^d)^n\\k_1 + \cdots + k_n = k }} \frac{1}{\langle k_1\rangle^{\eta}} \cdots \frac{1}{\langle k_n\rangle^{\eta}} \lesssim \frac{1}{\langle k\rangle^{n \eta - (n-1)d}}\;, 
			\end{equation*}
		\end{itemize}
	\end{lem}
	\begin{proof}
		The proof follows from elementary calculus.
	\end{proof}
	
	\begin{remarque}
		For $(\mathrm{iv})$, we only need $\eta > \frac{(n-1)d}{n}$ for the left hand side above to be summable, while $\eta < d$ is needed in order for the exponent of $\langle k\rangle$ to be $-n \eta + (n-1)d$. 
	\end{remarque}

	\begin{cor}\label{cor:convolution2} 
		Let $l\in \N$ and $1\leq j\leq l$. Then for all $\alpha\in(0,1), 0<\eps\ll 1$ and $\gamma=l(1-\alpha)+\eps$, there exists $C_{\alpha,\eps}>0$, such that for any $N<M$, we have
		\begin{align*}
			&\sum_{\substack{N< |k_1|,\cdots,|k_j|\leq M\\
		|k_{j+1}|,\cdots,|k_l|\leq M 	
		} }\frac{1}{\langle k_1+\cdots+k_l\rangle^{2\gamma}} \prod_{i=1}^{l}\frac{1}{\langle k_i\rangle^{2\alpha}}
		 \leq C_{\alpha,\eps}N^{-2\eps}.
		\end{align*}
	\end{cor}
	\begin{proof}
	Thanks to the assumption on $\alpha$ and $\gamma$, we repeatedly applied (i) in Lemma \ref{lem:convolutionBasic} $l-1$ times, we control the desired summation by
	$$ \sum_{N<|k_1|\leq M}\langle k_1\rangle^{-2(\alpha+\gamma-(l-1)(1-\alpha) )}= \sum_{N<|k_1|\leq M}\langle k_1\rangle^{-2(1+\eps)}\lesssim N^{-2\eps}. 
	$$
	This completes the proof.
	\end{proof}

	\section{Large deviation estimates} \label{appendix:largedeviation}

	\begin{lem}\label{continuity} 
		Let $\xi$ be a random process of the form
		such that for any $s,t\in\R$, $\xi(s,\cdot)$ and $\xi(t,\cdot)$ have the same law that is stationary in $x\in\T^2$. Assume that
		for some $\gamma\in\R$, $(\langle\nabla\rangle^{\gamma}\xi)(t,x)$ belongs to  $\mathcal{H}_{\leq l}$, the space of Wiener chaos of degree less than $l$, and moreover
		$$ \sup_{t\in\R}\mathbb{E}[\|\xi(t,x)\|_{H^{\gamma}(\T^2)}^2]\leq A^2
		$$
		for some $A>0$.
		Then for any $\gamma_1<\gamma$, there exist $C_{\gamma,q,r},c_{\gamma,q,r}>0$, such that for all $\lambda>1$, $T>1$ and $q\geq 2$, $r\geq 2$,
		\begin{align}\label{eq:continuity1} 
			\mathbb{P}\Big[\|\xi\|_{L_t^qW_x^{\gamma,r}([0,T]\times\T^2))}> \lambda \Big]\leq C_{\gamma,q,r}\exp\big(-c_{\gamma,q,r}T^{-\frac{2}{ql}}A^{-\frac{2}{l}}\lambda^{\frac{2}{l}}\big).
		\end{align}
	\end{lem}
	\begin{proof}
For any $p\geq q,r$, by Chebyshev,
		\begin{align*}
			\mathbb{P}\big[\|\xi\|_{L_t^qW_x^{\gamma,r}([0,T]\times\T^2))}> \lambda \big]\leq & \frac{C_{\gamma,\gamma_1}^p}{\lambda^p}\mathbb{E}[\|\xi\|_{L_t^qW_x^{\gamma,r}([0,T]\times\T^2) }^p].
		\end{align*}
		By Minkowski,
		$$	\big[\mathbb{E}\|\xi(t)\|_{L_t^qW_x^{\gamma,r}([0,T]\times\T^2)}^p\big]^{\frac{1}{p}}\leq \|\langle\nabla\rangle^{\gamma}\xi(t,x)\|_{L_t^qL_x^rL_{\omega}^p}
		$$
		Since for fixed $t,x$ $(\langle\nabla\rangle^{\gamma}\xi)(t,x)\in\mathcal{H}_{\leq l}$
		and $\xi(t)$ is stationary in space and time, by Proposition \ref{Wienerchaos}, we deduce that
		$$ \|\langle\nabla\rangle^{\gamma}\xi(t,x)\|_{L_t^qL_x^rL_{\omega}^p}
		\leq Cp^{\frac{l}{2}}T^{\frac{1}{q}}\sup_{t\in\R}\mathbb{E}\big[\|\xi(t)\|_{H_x^{\gamma}}^2]^{\frac{1}{2}}\leq Cp^{\frac{l}{2}}T^{\frac{1}{q}}A.
		$$
		Therefore,
		\begin{align*}
			\mathbb{P}\big[\|\xi\|_{L_t^qW_x^{\gamma,r}([0,T]\times\T^2))}> \lambda \big]\leq & \frac{C_{\gamma,\gamma_1}^p p^{\frac{lp}{2}}T^{\frac{p}{q}}A^p }{\lambda^p}.
		\end{align*}
		By optimizing the choice of $p$,
		the proof of Lemma \ref{continuity} is now complete.
	\end{proof}



	\begin{lem} \label{le:Fourier_criteria}
		If $\Xi$ be a stationary random distribution on $\T^d$ and belongs to Wiener chaos of order $n$. Let $\{\widehat{\Xi}(k)\}_{k \in \Z^d}$ denote its Fourier coefficients. If there exists $\gamma \in \R$ and $C_0 > 0$ such that
		\begin{equation*}
			\E |\widehat{\Xi}(k)|^{2} \leq C_o \langle k\rangle^{-d+2\gamma}
		\end{equation*}
		for every $k \in \Z^d$, then for every $\sigma> \gamma$ and every $q \in [1,+\infty)$, we have
		\begin{equation*}
		\E \|\Xi\|_{H^{-\sigma}(\T^d)}^{2}+	\E \|\Xi\|_{\cC^{-\sigma}(\T^d)}^{q} \leq C\;, 
		\end{equation*}
		where $C$ depends on $q,n,\gamma,\sigma,d$ and $C_0$ only. In particular, the bound is uniform in the class of stationary processes in order $n$ that satisfies the above bound for Fourier coefficients. 
	\end{lem}

	\begin{lem} \label{le:Wick_bd}
		Recall that
		\begin{equation*}
			W_{N} \stackrel{law}{=} \sum_{k \in \Z^2,\; |k|\leq N} \frac{g_k }{\langle k\rangle^{\alpha}} \mathrm{e}^{ik\cdot x}
		\end{equation*}
		is the fractional Gaussian field on $\T^2$, where $\rho\in \mathcal{S}(\R)$. If $\sigma>0$, $\theta \geq 0$ and $n \in \N$ satisfy
		\begin{equation*}
			\sigma \wedge 1 > n(1-\alpha)-\theta>0\;, 
		\end{equation*}
		then for every $q \in [1,+\infty)$, there exists $C=C(q,n,\alpha,\sigma,\theta)$ such that
		\begin{equation*}
		\sup_{N} \Big( N^{-2 \theta} \E \|W_{N}^{\diamond n}\|_{H^{-\sigma}}^{2} \Big)+	\sup_{N} \Big( N^{-q \theta} \E \|W_{N}^{\diamond n}\|_{\cC^{-\sigma}}^{q} \Big) < C\;.
		\end{equation*}
		As a consequence, the same is true when $\cC^{-\sigma}$ is replaced by $W^{-\sigma,p}$ for every $p \geq 1$. 
	\end{lem}
	\begin{proof}
		Without loss of generality, we can restrict to the situation where $n(1-\alpha) - \theta > 0$. Also, since $W_N^{\diamond n}$ belongs to Wiener chaos of order $n$, it suffices to prove for $q=2$. By explicit computation, we have
		\begin{equation*}
			\begin{split}
				\E |\widehat{W_N^{\diamond n}}(k)|^{2} &= \E \big( \sum_{\stackrel{k_1 + \cdots + k_n}{=k}} \widehat{W_N}(k_1) \diamond \cdots \diamond \widehat{W_N}(k_n) \big) \cdot \big( \sum_{\stackrel{\ell_1 + \cdots + \ell_n}{=k}} \widehat{W_N}(\ell_1) \diamond \cdots \diamond \widehat{W_N}(\ell_n) \big)\\
				&= n! \sum_{\stackrel{k_1 + \cdots + k_n}{=k}} \E |\widehat{W_N}(k_1)|^{2} \cdots \E |\widehat{W_N}(k_n)|^{2}\\
				&\lesssim \sum_{\stackrel{k_1 + \cdots + k_n}{=k}} \frac{1}{\langle k_1\rangle^{2\alpha}\cdots\langle k_n\rangle^{2\alpha} }\mathbf{1}_{|k_1|\leq N}\cdots\mathbf{1}_{|k_n|\leq N}\;. 
			\end{split}
		\end{equation*}
	We have
		\begin{equation*}
			N^{-2 \theta} \E |\widehat{W_N^{\diamond n}}(k)|^{2} \lesssim_{n} \sum_{\stackrel{k_1 + \cdots + k_n}{=k}} \prod_{j=1}^{n} \frac{1}{N^{\frac{2\theta}{n}} \big( 1 + |k_j|^{2\alpha}  \big)}\mathbf{1}_{|k_j|\leq N}\;, 
		\end{equation*}
		Hence, we get
		\begin{equation*}
			N^{-2 \theta} \E |\widehat{W_N^{\diamond n}}(k)|^{2} \lesssim \sum_{\stackrel{k_1 + \cdots + k_n}{=k}} \frac{1}{\langle k_1\rangle^{2 \alpha + \frac{2\theta}{n}}} \cdots \frac{1}{\langle k_n\rangle^{2 \alpha + \frac{2\theta}{n}}}\;. 
		\end{equation*}
		By (iv) of Lemma~\ref{lem:convolutionBasic}, if $\frac{2(n-1)}{n} < 2 \alpha + \frac{2\theta}{n} < 2$, we have the bound
		\begin{equation*}
			N^{-2\theta} \E |\widehat{W_N^{\diamond n}}(k)|^{2} \lesssim \langle k\rangle^{-2(\theta+1-n(1-\alpha))} = \langle k\rangle^{-2 + 2 (n(1-\alpha)-\theta }\;. 
		\end{equation*}
		Note that the above requirement is equivalently to our assumption
		\begin{equation*}
		0 < n(1-\alpha) - \theta < 1\;. 
		\end{equation*}
		 Now by Lemma~\ref{le:Fourier_criteria}, if
		\begin{equation*}
			\sigma > n(1-\alpha) - \theta\;, 
		\end{equation*}
		the desired bound follows. We have thus completed the proof of the Lemma \ref{le:Wick_bd}. 
	\end{proof}			
	Now we provide the proof of Lemma \ref{largedeviation}:
	\begin{proof}[Proof of Lemma \ref{largedeviation}]
		
		From the Sobolev embedding $W^{-\beta-\frac{\eps}{2},\frac{8}{\eps}}\hookrightarrow W^{-\beta-\eps,\infty}$ and Lemma \ref{newlargedeviation}, we deduce that there exist $C>0$ and  $\delta=\delta(\eps)>0$, such that
		$$ \mu(\Sigma_{2,N}^c)+\mu(\Sigma_{5,N}^c)<Ce^{-\delta(\eps)}.
		$$
		To estimate $\mu(\Sigma_{i,N}^c)$ for $i=1,3,4$, by Lemma \ref{continuity}, it suffices to show that for all $k\geq 4, 0\leq l\leq 3$, and $N\leq N_1\leq N$, we have the following estimates:
		\begin{align}
			& \mathbb{E}[\|\<1>_N^{\diamond l}-\<1>^{\diamond l}\|_{H^{-l\beta-\eps}}^2]+\mathbb{E}[\|\<1>_N^{\diamond l}-\<1>_{N_1}^{\diamond l}\|_{H^{-l\beta-\eps}}^2]\lesssim_{\eps} N^{-2\eps},\label{B2} \\
			&\mathbb{E}[\|\<1>_N^{\diamond l}\|_{H^{-l\beta-\eps}}^2]\lesssim_{\eps}1,\label{B3} \\
			& \mathbb{E}[\|\<1>_N^{\diamond k}\|_{H^{\frac{\eps}{2}}}^2]\lesssim_{\eps} N^{2k\beta+\frac{3\eps}{2}},\label{B4} \\
			&\mathbb{E}[\|\<1>_N^{\diamond (k-l)} \|_{H^{-(3-l)\beta-2\eps}}^2]\lesssim_{\eps} N^{2(k-3)\beta-\eps}. \label{B5}
		\end{align}
Note that \eqref{B3},\eqref{B4},\eqref{B5} are consequences of Lemma \ref{le:Wick_bd}, hence it remains to prove \eqref{B2}. 

Let $M> N$ and denote by $\gamma=l\beta+\eps$. Note that under the law $\mu$, $\<1>_N^{\diamond l}-\<1>_{M}^{\diamond l}$ is the same as $\phi_N^{\diamond l}-\phi_{M}^{\diamond l}$. Denote by
$$ \phi_{N,M}:=\sum_{N<|k|\leq M}\frac{g_k(\omega)}{\sqrt{1+|k|^{\alpha}}}\mathrm{e}^{ik\cdot x},\quad  \widetilde{\sigma}_{N,M}^2:=\widetilde{\sigma_M}^2-\widetilde{\sigma}_N^2.$$
Using the white noise functional representation as in Section 2,
$$ \phi_N(x)=\widetilde{\sigma}_NW_{\eta_N(x,\cdot)},\quad \phi_{N,M}(x)=\widetilde{\sigma}_NW_{\eta_{N,M}(x,\cdot)},
$$
where
$$ \eta_N(x,\cdot)=\sum_{|k|\leq N}\frac{1}{\sqrt{1+|k|^{2\alpha}}}\mathrm{e}^{ik\cdot(x+\cdot)},\quad  \eta_N(x,\cdot)=\sum_{N<|k|\leq M}\frac{1}{\sqrt{1+|k|^{2\alpha}}}\mathrm{e}^{ik\cdot(x+\cdot)}.
$$
Combining \eqref{binomHermite}, we can write
\begin{align*}
\phi_N^{\diamond l}-\phi_M^{\diamond l}=\sum_{j=1}^l\binom{l}{j}\widetilde{\sigma}_{N,M}^j\widetilde{\sigma}_N^{l-j}H_j\big(W_{\eta_N(x,\cdot)} \big) H_{l-j}\big(W_{\eta_{N,M}(x,\cdot)}\big).
\end{align*} 
Using \eqref{esperance-scalaire} and the independence of $W_{\eta_N(x,\cdot)}, W_{\eta_{N,M}(x,\cdot)}$, for $j\in\{1,\cdots,l\}$, we have
\begin{align*}
&\widetilde{\sigma}_{N,M}^{2j}\widetilde{\sigma}_N^{2(l-j)}\mathbb{E}\big[\big\|H_j(W_{\eta_N(x,\cdot)})H_{l-j}(W_{\eta_{N,M}(x,\cdot)}) \big\|_{H^{-\gamma}}^2\big]\\
\sim & \sum_{N<\substack{|k_1|,\cdots,|k_j|\leq M,\\
|k_{j+1}|,\cdots,|k_{l}|,|k|\leq M		
		\\
k_1+\cdots+k_l+k=0 } }\frac{1}{\langle k\rangle^{2\gamma}}\prod_{j=1}^l\frac{1}{\langle k_j\rangle^{2\alpha} }\lesssim N^{-2\eps},
\end{align*} 
thanks to Corollary \ref{cor:convolution2}. The proof of Lemma \ref{largedeviation} is now complete.
	\end{proof} 

	\section{Proof of the Strichartz estimate on $\T^d$} \label{appendix:strichartz}

	\begin{lem}\label{dispersive}
		Let $K_j^{\pm}(t,x-y)$ be the Schwartz kernel of the operator $e^{\pm it\dD^{\alpha}}\mathbf{P}_j$, $j\geq 0$. Then for any $t\neq 0$,
		\begin{align}\label{kernel}  \sup_{z\in\T^d}|K_j^{\pm}(t,z)|\lesssim \frac{2^{jd\big(1-\frac{\alpha}{2}\big)}}{|t|^{\frac{d}{2}}}. 
		\end{align}
		Consequently, for $t\geq 0$ and $2\leq r\leq\infty$,
		\begin{align}\label{eq:dispersive} \|e^{\pm it \dD^{\alpha}}\mathbf{P}_jf \|_{L^{r}(\T^d)}\lesssim \frac{2^{jd\big(1-\frac{\alpha}{2}\big)\big(1-\frac{2}{r}\big) }}{|t|^{\frac{d}{2}\big(1-\frac{2}{4}\big) } }\|\mathbf{P}_jf\|_{L^{r'}(\T^d)},
		\end{align}	
		where $r'$ is such that $\frac{1}{r}+\frac{1}{r'}=1$.
	\end{lem}		
	\begin{proof}
		The kernel $K_j^{\pm}(t,z)$ takes the form
		$$ K_j^{\pm}(t,z)=\sum_{k\in\Z^d}\varphi_j(k)e^{it\sqrt{1+|k|^{2\alpha}}+ik\cdot z}.
		$$
		From the Poisson summation formula, we have
		\begin{align}\label{Poisson} 
			K_j^{\pm}(t,z)=&(2\pi)^d\sum_{m\in\Z^d}\mathcal{F}_{\R^d}^{-1}(\varphi_j(\cdot)e^{\pm it\sqrt{1+|\cdot|^{2\alpha})}}(z+m)\notag \\
			=&2^{jd}\sum_{m\in\Z^d}\kappa_{j,m}^{\pm}(t,z).
		\end{align}
		where
		$$ \kappa_{j,m}^{\pm}(t,z):=\int_{\R^d}\varphi(\xi)e^{\pm it\sqrt{1+|2^j\xi|^{2\alpha}}+i\xi\cdot 2^j(z+m)}d\xi.
		$$
		Consider the phase function
		$$ \Phi_{t,z,m}^{\pm}(\xi):=\pm\sqrt{2^{-2j\alpha}+|\xi|^{2\alpha}}+2^{j(1-\alpha)}(z+m)\cdot\xi,
		$$
		then $\kappa_{j,m}^{\pm}(t,z)=\mathcal{I}_{z,m}(2^{j\alpha}t)$, where
		$$ \mathcal{I}_{z,m}(\lambda t):=\int_{\R^d}\varphi(\xi)e^{i\lambda t \Phi^{\pm}_{t,z,m}(\xi)}d\xi. 
		$$
		Note that
		\begin{align*} &\nabla_{\xi}\Phi_{t,z,m}^{\pm}(\xi)=\pm \alpha|\xi|^{2\alpha-2}\frac{\xi}{\sqrt{2^{-2j\alpha}+|\xi|^{2\alpha} }}+2^{j(1-\alpha)}(z+m)
		\end{align*}
		and on supp$(\varphi)$, 
		$$ |\nabla_{\xi}\Phi_{t,z,m}^{\pm}|\gtrsim 1+2^{j(1-\alpha)}|m|, \;\forall |m|\geq 2.
		$$
		Moreover, on supp$(\varphi)$, $|\det(\nabla_{\xi}^2\Phi_{t,z,m}^{\pm}(\xi))|\gtrsim 1$. By the stationary phase lemma, we have
		$$ |\mathcal{I}_{z,m}(\lambda t)|\lesssim \frac{1}{|\lambda t|^{\frac{d}{2}}},\quad |m|\leq 2
		$$
		and
		$$ |\mathcal{I}_{z,m}(\lambda t)|\lesssim \frac{C_N}{|\lambda t|^N(1+2^{j(1-\alpha)|m|})^N },\quad  |m|\geq 2
		$$
		for all $N\in\N$. 
		Plugging into \eqref{Poisson}, we obtain \eqref{kernel}.
		
		Replacing $e^{\pm it\dD^{\alpha}}\mathbf{P}_j$ by $e^{\pm it\dD^{\alpha}}\widetilde{\mathbf{P}}_j$, where $\widetilde{\mathbf{P}}_j$ is a similar Littlewood-Paley projector such that $\widetilde{\mathbf{P}}_j\mathbf{P}_j=\mathbf{P}_j$, the same kernel estimate holds for $e^{\pm it\dD^{\alpha}}\widetilde{\mathbf{P}}_j$.
		Consequently, we have
		$$ \|e^{\pm it\dD^{\alpha}}\mathbf{P}_jf\|_{L^{\infty}(\T^d)}\lesssim \frac{2^{jd\big(1-\frac{\alpha}{2}\big) }}{|t|^{\frac{d}{2}}}\|\mathbf{P}_jf\|_{L^1(\T^d)}.
		$$
		Note that $e^{\pm it\dD^{\alpha}}$ is an isometry on $L^2(\T^d)$, applying the Riesz-Thorin interpolation theorem, we deduce that for all $2\leq r\leq \infty$,
		$$
		\|e^{\pm it \dD^{\alpha}}\mathbf{P}_jf \|_{L^{r}(\T^d)}\lesssim \frac{2^{jd\big(1-\frac{\alpha}{2}\big)\big(1-\frac{2}{r}\big) }}{|t|^{\frac{d}{2}\big(1-\frac{2}{4}\big) } }\|\mathbf{P}_jf\|_{L^{r'}(\T^d)},
		$$
		and this completes the proof.
	\end{proof}		
	Now we are able to prove Proposition \ref{Strichartz1}. The solution $u$
	\begin{align}\label{linearwave} 
		\partial_t^2u+(\dD^{\alpha})^2u=F,\quad (u,\partial_tu)|_{t=0}=(u_0,u_1)
	\end{align} 
	can be written as
	$$ u(t)=\cos(t\dD^{\alpha})u_0+\frac{\sin(t\dD^{\alpha})}{\dD^{\alpha}}u_1+\int_0^t\frac{\sin((t-t')\dD^{\alpha})}{\dD^{\alpha}}F(t')dt'.
	$$
	It suffices to prove the homogeneous estimate
	\begin{align}\label{homoStrichartz} 
		\|e^{\pm it\dD^{\alpha}}f\|_{L_t^qL_x^r(\R\times\T^d)}\lesssim \|f\|_{H^{\gamma_{q,r}}(\T^d)}
	\end{align}
	and the following inhomogeneous estimate
	\begin{align}\label{inhomoStrichartz} 
		\Big\|\int_{\R}e^{\pm i(t-t')\dD^{\alpha}}G(t')dt' \Big\|_{L_t^qL_x^r(\R\times\T^d)}\lesssim \|G\|_{L_t^1H^{\gamma_{q,r}}(\R\times\T^d)},
	\end{align}
	thanks to the Christ-Kiselev Lemma (\cite{CK}).
	
	We perform a standard $TT^*$ argument. 	Fix a sharp admissible pair $(q,r)$, i.e.
	$$ \frac{2}{q}=d\big(\frac{1}{2}-\frac{1}{r}\big),\quad (q,r,d)\neq (2,\infty,2),
	$$ define
	\begin{align}
		&\mathcal{T}_j: L_x^2\rightarrow L_t^qL_x^r, \quad \mathcal{T}_j(f):=e^{\pm it\langle\nabla \rangle^{\alpha}}\mathbf{P}_jf,\\ &\mathcal{T}_j^*: L_t^{q'}L_x^{r'}\rightarrow L_x^2,\quad  \mathcal{T}_j^*G:=\int_{\R}e^{\mp it\langle\nabla\rangle^{\alpha}}\mathbf{P}_jG(t)dt.
	\end{align}
	Using \eqref{dispersive} and the Hardy-Littlewood-Sobolev inequality, we deduce that
	\begin{align*}
		\Big\|\int_{\R}e^{\pm i(t-t')\dD^{\alpha} }\mathbf{P}_jG(t')dt'\Big\|_{L_t^{q}L_x^r}\lesssim 2^{jd\big(1-\frac{\alpha}{2}\big)\big(1-\frac{2}{r}\big) }\|\mathbf{P}_jG\|_{L_t^{q'}L_x^{r'}},
	\end{align*}
	Since
	$$ \|\mathcal{T}_j\|_{L_x^2\rightarrow L_t^qL_x^r}=\|\mathcal{T}_j^*\|_{L_t^{q'}L_x^{r'}\rightarrow L^2}=\|\mathcal{T}_j\mathcal{T}_j^*\|_{L_t^{q'}L_x^{r'}\rightarrow L_t^qL_x^r}^{1/2},
	$$		
	we deduce further that for any admissible pairs $(q_1,r_1),(q,r)$,
	$$  \|\mathcal{T}_j\|_{L_x^2\rightarrow L_t^qL_x^r}\lesssim 2^{jd\big(1-\frac{\alpha}{2}\big)\big(\frac{1}{2}-\frac{1}{r}\big)}.
	$$
	Therefore, for any admissible pairs $(q_1,r_1),(q,r)$,
	$$ \|\mathcal{T}_j\mathcal{T}_j^*\|_{L_t^{q_1'}L_x^{r_1'}\rightarrow L_t^qL_x^r }\lesssim 2^{jd\big(1-\frac{\alpha}{2}\big)\big(1-\frac{1}{r}-\frac{1}{r_1}\big) }.
	$$
	In particular, for $q_1=\infty, r_1=2$, we have
	$$ \|e^{\pm it\dD^{\alpha}}\mathbf{P}_jf\|_{L_t^qL_x^r}+\Big\|\int_{\R}e^{\pm i(t-t')\dD^{\alpha}}\mathbf{P}_jG(t')dt'\Big\|_{L_t^qL_x^r}\lesssim 2^{j\gamma_{q,r}}\|\mathbf{P}_jf\|_{L_x^2}+2^{j\gamma_{q,r}}\|\mathbf{P}_jG\|_{L_t^1L_x^2}.
	$$
	Taking the $l^2$ norm in $j$, we obtain that
	$$  \|e^{\pm it\dD^{\alpha}}\mathbf{P}_jf\|_{l_j^2L_t^qL_x^r}+\Big\|\int_{\R}e^{\pm i(t-t')\dD^{\alpha}}\mathbf{P}_jG(t')dt'\Big\|_{l_j^2L_t^qL_x^r}\lesssim \|f\|_{H_x^{\gamma_{q,r}}}+\|G\|_{L_t^1H_x^{\gamma_{q,r}}}
	$$
	Since $2\leq r<\infty, q\geq 2$, by the Minkowski inequality and the Littlewood-Paley square function theorem, 
	$$ \|F\|_{L_t^qL_x^r}\sim \|\mathbf{P}_jF\|_{L_t^qL_x^rl_j^2}\leq \|F\|_{l_j^2L_t^qL_x^r},
	$$ 			
	thus we have proved \eqref{homoStrichartz} and \eqref{inhomoStrichartz}. This completes the proof of Proposition \ref{Strichartz1}.

	\section{Convergence of the linear coefficient} \label{appendix:1orderconstant}
	In this section, we prove Proposition \ref{convergeneclinear}. 
	Recall that $\widetilde{\sigma}_N^2=N^{2(1-\alpha)}\sigma_N^2$. 
			$$ \widetilde{\sigma}_N^2-N^{2(1-\alpha)}\sigma^2=\frac{1}{4\pi^2}\sum_{|k|\leq N}\frac{1}{1+|k|^{2\alpha}}-\frac{1}{4\pi^2}\int_{|\xi|\leq N}\frac{1}{|\xi|^{2\alpha}}d\xi.
			$$
			In order to prove the convergence of $N^{2(1-\alpha)}(\ov{a}_{1,N}-\ov{a}_1)$, the key is to show that:
			\begin{lem}\label{asymptotic} 
				Assume that $\alpha\in(\frac{1}{2},1),$ Then
				$$ \sigma_N^2=\sigma^2+\ov{b}_1N^{-2(1-\alpha)}+O(N^{-1}).
				$$
				  where
				 $$\ov{b}_1=\frac{1}{4\pi^2}+\frac{1}{4\pi^2}\sum_{k\in\Z^2}\int_{C_k}\Big(\mathbf{1}_{k\neq 0}\frac{1}{1+|k|^{2\alpha}}-\frac{1}{|\xi|^{2\alpha}}\Big)d\xi,
				$$
				where $(C_k)_{k\in\Z^2}$ are unit cubes $[k^{(1)},k^{(1)}+1]\times[k^{(2)},k^{(2)}+1]$. 
			\end{lem}
			
			\begin{proof}
			We denote 
				\begin{align*}
					N^{2(1-\alpha)}\big(\sigma_N^2-\sigma^2\big)=\frac{1}{4\pi^2}\big(1+\mathrm{I}_N\big),
				\end{align*}
				where
				\begin{align}\label{I} 
					\mathrm{I}_N:=\sum_{0<|k|\leq N}\frac{1}{1+|k|^{2\alpha}}-\int_{|\xi|\leq N}\frac{1}{|\xi|^{2\alpha}}d\xi.
				\end{align}
				We decompose
				\begin{align*}
					Z_N:=\mathbb{Z}^2\cap \{k:0<|k|\leq N \}=\bigcup_{j=1}^8\Lambda_j, \quad B_N:=\{\xi:|\xi|\leq N \}=\bigcup_{j=1}^4\ov{U}_j
				\end{align*}
				where
				\begin{align*}
					&\Lambda_1:=\{k=(k^{(1)},k^{(2)})\in Z_N: k^{(1)}\geq 0,k^{(2)}\geq 0 \},\quad U_1:=\{\xi=(\xi^{(1)},\xi^{(2)})\in B_N, \xi^{(1)}>0,\xi^{(2)}>0  \}, \\
					&\Lambda_2:=\{k=(k^{(1)},k^{(2)})\in Z_N: k^{(1)}\leq 0,k^{(2)}\geq 0 \},\quad
					U_2:=\{\xi=(\xi^{(1)},\xi^{(2)})\in B_N: \xi^{(1)}<0,\xi^{(2)}>0 \},
					\\
					&\Lambda_3:=\{k=(k^{(1)},k^{(2)})\in Z_N: k^{(1)}\leq 0,k^{(2)}\leq 0 \},
					\quad
					U_3:=\{\xi=(\xi^{(1)},\xi^{(2)})\in B_N: \xi^{(1)}<0,\xi^{(2)}<0 \},
					\\
					&\Lambda_4:=\{k=(k^{(1)},k^{(2)})\in Z_N: k^{(1)}\geq 0,k^{(2)}\leq 0 \},\quad
					U_4:=\{\xi=(\xi^{(1)},\xi^{(2)})\in B_N: \xi^{(1)}>0,\xi^{(2)}<0 \}.
				\end{align*}
				For $j=1,2,3,4$, we define
				$$ \mathrm{I}_{N,j}:=\sum_{\Lambda_j}\frac{1}{1+|k|^{2\alpha}}-\int_{U_j}\frac{1}{|\xi|^{2\alpha}}d\xi.
				$$
				Then by inclusion and exclusion,
				\begin{align}\label{IN} 
				\mathrm{I}_N=\sum_{j=1}^4\mathrm{I}_{N,j}-\sum_{\substack{k^{(1)}k^{(2)}=0\\
				0<|k|\leq N } }\frac{1}{1+|k|^{2\alpha}}.
				\end{align}
				By symmetry, it suffices to derive a formula for $\mathrm{I}_{N,1}$. 
				Fix $k=(k^{(1)},k^{(2)})\in\Lambda_1$, we denote 
				$$ C_k:=\{\xi=(\xi^{(1)},\xi^{(2)}): k^{(j)}\leq \xi^{(j)}\leq k^{(j)}+1, j=1,2 \},$$
				and
				$$ C_{0}^{(1)}:=\{k=(k^{(1)},k^{(2)}):0\leq k^{(1)},k^{(2)}\leq 1  \}
				$$
				the cubic with bottom left vertex $k$ and top right vertex $\theta(k):=(k^{(1)}+1, k^{(2)}+1)$. 
				We have
				\begin{align*}
					\int_{U_1}\frac{d\xi}{|\xi|^{2\alpha}}=\int_{C_0^{(1)}}\frac{d\xi}{|\xi|^{2\alpha}}+\sum_{k\in \Lambda_1}\int_{C_k\cap B_N}\frac{d\xi}{|\xi|^{2\alpha}}.
				\end{align*}
				Let
				$$ \widetilde{U}_{1}:=U_1\cup \bigcup_{k\in\Lambda_1}C_k
				$$
				Since the number of cubes $C_k$ intersecting with $|k|=N$ is $O(N)$, we have 
				$$ \int_{\widetilde{U}_1\setminus U_1}\frac{d\xi}{|\xi|^{2\alpha}}=O(N^{-(2\alpha-1)}).
				$$
				Thus
				\begin{align*}
					\mathrm{I}_{N,1}:=-\int_{C_0^{(1)}}\frac{d\xi}{|\xi|^{2\alpha}}+\sum_{k\in\Lambda_1}\int_{C_k}\Big(\frac{1}{1+|k|^{2\alpha}}-\frac{1}{|\xi|^{2\alpha}}\Big)d\xi+O(N^{-(2\alpha-1)}).
				\end{align*}
				We have similar formulas for $\mathrm{I}_{N,2},\mathrm{I}_{N,3},\mathrm{I}_{N,3}$. Adding them together and noticing that the lattices on two axes have been added twice, we have
				\begin{align*}
					\sum_{j=1}^4\mathrm{I}_{N,j}=&-\int_{C_0^{(1)}\cup C_0^{(2)}\cup C_0^{(3)}\cup C_0^{(4)}}\frac{d\xi}{|\xi|^{2\alpha}}+\sum_{k\in Z_N}\int_{C_k}\Big(\frac{1}{1+|k|^{2\alpha}}-\frac{1}{|\xi|^{2\alpha}}\Big)d\xi+\sum_{\substack{k^{(1)}k^{(2)}=0\\
				0<|k|\leq N	
				 } }\frac{1}{1+|k|^{2\alpha}}\\+&O(N^{-(2\alpha-1)}).
				\end{align*}
				Since $\alpha>\frac{1}{2}$, we have
				$$ \sum_{k\in Z_N}\int_{C_k}\Big(\frac{1}{1+|k|^{2\alpha}}-\frac{1}{|\xi|^{2\alpha}}\Big)d\xi=\sum_{k\neq 0}\int_{C_k}\Big(\frac{1}{1+|k|^{2\alpha}}-\frac{1}{|\xi|^{2\alpha}}\Big)d\xi+O(N^{1-2\alpha}), 
				$$
				and
				$$ \sum_{\substack{k^{(1)}k^{(2)}=0\\
						0<|k|\leq N	
				} }\frac{1}{1+|k|^{2\alpha}}=2\sum_{m\neq 0,m\in\N}\frac{1}{1+|m|^{2\alpha}}+O(N^{1-2\alpha}).
				$$
			We have 
				\begin{align*}
				\sum_{j=1}^4\mathrm{I}_{N,j}=\sum_{k\in\Z^2}\int_{C_k}\Big(\mathbf{1}_{k\neq 0}\frac{1}{1+|k|^{2\alpha}}-\frac{1}{|\xi|^{2\alpha}}\Big)d\xi+2\sum_{0\neq m\in\Z}\frac{1}{1+|m|^{2\alpha}}+O(N^{1-2\alpha}).
				\end{align*}
				Therefore,
				\begin{align*}
					\mathrm{I}_N=\sum_{k\in\Z^2}\int_{C_k}\Big(\mathbf{1}_{k\neq 0}\frac{1}{1+|k|^{2\alpha}}-\frac{1}{|\xi|^{2\alpha}}\Big)d\xi+O(N^{1-2\alpha}).
				\end{align*}
				This completes the proof.
			\end{proof}	
\begin{proof}[Proof of Proposition \ref{convergeneclinear}]
		By definition,
			\begin{align*}
				\ov{a}_{1,N}-\ov{a}_1=\frac{1}{2}\int_{-\infty}^{\infty}V''(z)\Big(\frac{1}{\sqrt{2\pi}\sigma_N}e^{-\frac{z^2}{2\sigma_N}}-\frac{1}{\sqrt{2\pi}\sigma}e^{-\frac{z^2}{2\sigma}}\Big)dz.
			\end{align*}
			By Lemma \ref{asymptotic} and the fact that $\alpha>\frac{3}{4}$, we get
			$$ N^{2(1-\alpha)}(\sigma_N-\sigma)=\frac{\ov{b}_1}{2\sigma}+\eps_N,$$
			where
			$$ 
			 \eps_N=O(N^{-(2\alpha-1)})+O(N^{-2(1-\alpha)}).
			$$
			 we finally obtain that
			\begin{align}\label{limita1N} 
			 N^{2(1-\alpha)}(\ov{a}_{1,N}-\ov{a}_1)=\frac{\ov{b}_1}{4\sigma}\int_{-\infty}^{\infty}\partial_{\sigma}\Big(\frac{1}{\sqrt{2\pi}\sigma}e^{-\frac{z^2}{2\sigma}}\Big)V''(z)dz+\eps_N.
			\end{align}
		This completes the proof of Proposition \ref{convergeneclinear}.
\end{proof}


\end{document}